\documentclass{amsart}

\usepackage[all]{xy}    
\usepackage{graphicx}
\usepackage{amssymb}
\usepackage{mathrsfs}
\usepackage{multirow}
\usepackage{float}
\usepackage{tikz-cd}
\usepackage{adjustbox}
\usepackage{amsthm}
\usepackage{accents}
\usepackage{mathtools}
\usepackage{mathptmx,enumitem,cite,array}
\usepackage[new]{old-arrows}
\usepackage{tikz}
\usetikzlibrary{matrix,arrows}
\usepackage{amsmath}

\newcommand{\ubar}[1]{\underaccent{\bar}{#1}}

\usepackage[numbers,sort&compress]{natbib} 
\usepackage[bookmarksnumbered, bookmarksopen,
colorlinks,citecolor=blue,linkcolor=blue,backref]{hyperref}

\newcounter{RomanNumber}
\newcommand{\MyRoman}[1]{\setcounter{RomanNumber}{#1}\Roman{RomanNumber}}

\newtheorem{theorem}{Theorem}
\newtheorem{lemma}[theorem]{Lemma}

\theoremstyle{definition}
\newtheorem{definition}[theorem]{Definition}
\newtheorem{example}[theorem]{Example}

\newtheorem{corollary}[theorem]{Corollary}

\newtheorem{remark}[theorem]{Remark}

\theoremstyle{notation}

\newcommand{\pullbackcorner}[1][dr]{\save*!/#1-1.2pc/#1:(1,-1)@^{|-}\restore}
\newcommand{\pushoutcorner}[1][dr]{\save*!/#1+1.2pc/#1:(-1,1)@^{|-}\restore}
\newcommand{\be}{\begin{equation}}
\newcommand{\ee}{\end{equation}}

\newcommand{\CC}{\mathbb{C}}

\newcommand{\R}{\mathbb{R}}

\newcommand{\Z}{\mathbb{Z}}

\usepackage[utf8]{inputenc}
\usepackage{chngcntr}
\usepackage{apptools}
\AtAppendix{\counterwithin{theorem}{section}}
\usepackage[toc,page]{appendix}

\begin{document}


\title
{String$^c$ Structures and Modular InvariantsString$^c$ Structures and Modular Invariants}

 \author{Ruizhi Huang}
\address{Institute of Mathematics and Systems Sciences, Chinese Academy of Sciences, Beijing 100190, China}
\email{huangrz@amss.ac.cn}  
\urladdr{https://sites.google.com/site/hrzsea \\ 
https://hrzsea.github.io/Huang-Ruizhi/}


\author{Fei Han}
\address{Department of Mathematics,
National University of Singapore, Singapore 119076}
\email{mathanf@nus.edu.sg}

\author{Haibao Duan}
\address{Institute of Mathematics and Systems Sciences, Chinese Academy of Sciences, Beijing 100190, China}
\email{dhb@math.ac.cn}

\date{}

\maketitle

\begin{abstract}
In this paper, we study some algebraic topology aspects of  String$^c$ structures, more precisely, from the perspective of Whitehead tower and the perspective of the loop group of $Spin^c(n)$. We also extend the generalized Witten genera constructed for the first time in \cite{CHZ11} to correspond to String$^c$ structures of various levels and give vanishing results for them. 
\end{abstract}

\tableofcontents


\section{Introduction}

\subsection{Background} \noindent Let $V$ be a real rank $n$ oriented vector bundle over a connected manifold $M$. Let $F_{SO(n)}$ be the oriented orthonormal frame bundle of $V$ over $M$. $V$ is called {\em Spin} if $F_{SO(n)}$ has an equivariant lift with respect to the double covering $\rho:Spin(n)\rightarrow SO(n).$ A Spin structure is a pair $(P, f_P)$ with $\pi_P: P\to M$ being a principal $Spin(n)$-bundle over $M$ and $f_P: P\to F_{SO(n)}$ being an equivariant 2-fold covering map. 
$P$ is called the {\em bundle of Spin frames} of $V$. The topological obstruction to the existence of Spin structure is the second Stiefel-Whitney $\omega_2(V)$. Furthermore, if it vanishes then the distinct Spin structures lifting the prescribed oriented structure on $V$ are in one-to-one correspondence with the elements in $H^1(M; \Z/2)$ (see \cite{Lawson89}). 

A {\em String structure} is a higher version of Spin structure, which is related to quantum anomaly in physics \cite{K87}. One mathematical way to look at String structures is from the perspective of Whitehead tower. A String group is an infinite-dimensional group $String(n)$ introduced by Stolz \cite{St96} as a 3-connected cover of $Spin(n)$. Let $V$ be a vector bundle with the Spin structure $(P, f_P)$. Let $g: M\to BSpin(n)$ be the classifying map of $P$. $V$ is called {\em (strong) String}, if there is a lift,
\begin{gather*}
\begin{aligned}
\xymatrix{
& BString(n) \ar[d]  \\
M \ar[r]_{ g  \ \ \ \ \ \ } \ar@{-->}[ru]  & BSpin(n).
}
\end{aligned}
\label{stringliftdiagintro}
\end{gather*}
The obstruction to the lift is $\frac{1}{2}p_1(V)$, and if it vanishes then the distinct String structures lifting the prescribed Spin structure on $V$ are in one-to-one correspondence with the elements in $H^3(M; \Z)$. 

Another way to look at String structure is from the perspective of free loop space $LM$, namely by looking at lifting of the structure group of the looped Spin frame bundle from the loop group to its universal central extension \cite{Mclau92}. Under this point of view, the obstruction to the existence of \textit{the (weak) String structure} is the transgression of $\frac{1}{2}p_1(V)$, and if it vanishes then the distinct String structures lifting the prescribed Spin structure on $V$ are in one-to-one correspondence with the elements of $H^2(LM; \Z)$. These two approaches to look at String structures are equivalent when $M$ is $2$-connected. In general strong String is strictly stronger than weak String.

More geometrically, Stolz and Teichner gave the profound link of the String structure on $M$ to the fusive Spin structure on $LM$ \cite{ST05}. This was further developed by Waldorf \cite{Wal15, Wal16} and Kottke-Melrose \cite{KM13}. In \cite{Bun}, Bunke studied the Pfaffian line bundle of a certain family of real Dirac operators and showed that String structures give rise to trivialisations of that Pfaffian line bundle. See also the study of String structures from the differential and the twisted point of view \cite{Re, Sati}.

Let $M$ be a $4m$ dimensional compact oriented smooth manifold. Let $\{\pm
2\pi \sqrt{-1}z_{j},1\leq j\leq 2m\}$ denote the formal Chern roots of $T_{%
\mathbb{C}}M $, the complexification of the tangent vector bundle $TM$ of $M$%
. Then the famous Witten genus of $M$ can be written as (cf. \cite{Liu96})%
\begin{equation*}
W(M)=\left\langle \left( \prod_{j=1}^{2m}z_{j}\frac{\theta ^{\prime }(0,\tau
)}{\theta (z_{j},\tau )}\right) ,[M]\right\rangle \in \mathbb{Q}[[q]],
\end{equation*}%
with $\tau \in \mathbb{H}$, the upper half-plane, and $q=e^{\pi \sqrt{-1}%
\tau }$. The Witten genus was first introduced in \cite{W} and can be viewed as the
loop space analogue of the $\widehat{A} $-genus. It can be expressed as a $q$%
-deformed $\widehat{A}$-genus as
\begin{equation*}
W(M)=\left\langle \widehat{A}(TM)\mathrm{ch}\left( \Theta \left( T_{\mathbb{C%
}}M\right) \right) ,[M]\right\rangle ,
\end{equation*}%
where
\begin{equation*}
\Theta (T_{\mathbb{C}}M)=\overset{\infty }{\underset{n=1}{\otimes }}%
S_{q^{2n}}(\widetilde{T_{\mathbb{C}}M}),\ \ {\rm with}\ \
\widetilde{T_{\mathbb{C}}M}=T_{\mathbb{C}}M-{\mathbb C}^{4m},
\end{equation*}%
is the Witten bundle defined in \cite{W}. When the manifold $M$ is Spin, according to the
Atiyah-Singer index theorem \cite{AS},  the Witten genus can be expressed analytically as the index of twisted Dirac operators, $W(M)=\mathrm{ind}(D\otimes \Theta \left( T_{\mathbb{C%
}}M\right))\in \mathbb{Z}[[q]]$, where $D$ is the Atiyah-Singer Spin Dirac operator on $M$ (cf. \cite{HBJ}). Moreover, if $M$ is String,  i.e. $\frac{1}{2}p_{1}(TM)=0$, or even weaker, if $M$ is Spin and the
first rational Pontryagin class of $M$ vanishes, then $W(M)$ is a
modular form of weight $2k$ over $SL(2,\mathbb{Z})$ with integral
Fourier expansion (\cite{Za}). The homotopy theoretical
refinements of the Witten genus on String manifolds leads to the
theory of tmf ({\em topological modular form}) developed by Hopkins and
Miller \cite {Hop}. The String condition is the orientablity condition for this generalized cohomology theory. 

As one of the important applications, the Witten genus can be used as obstruction to continuous symmetry on manifolds. In \cite{Liu95}, Liu discovered a profound vanishing theorem for the Witten genus under the anomaly condition that $p_1(M)_{S^1}=n\cdot \pi^* u^2$, where
$p_1(M)_{S^1}$ is the equivariant first Pontrjagin class, $\pi:
M\times_{S^1}ES^1\rightarrow BS^1$ is the projection from the Borel space to the classifying space and $u\in
H^2(BS^1,\mathbb{Z})$ is a generator and $n$ is an integer. Dessai \cite{Des94} showed that when the $S^1$-action is induced from an $S^3$-action and the manifold is String, this anomaly condition holds. Liu's vanishing theorem has been generalized in \cite{LM1, LM2, LMZ1, LMZ2, LY} for the family case, in \cite{LMZ3} for the foliation case and recently in \cite{HM18} for proper actions of non-compact Lie groups on non-compact manifolds.

Spin$^c$ structure is the complex analogue of Spin structure.
It is known that there exists a Spin structure on a vector bundle $V$ over $M$ if and only if its second Stiefel-Whitney class $w_2(V)=0$. In contrast, if $w_2(V)$ is only assumed to be trivial after applying the Bockstein, or equivalently, $w_2(V)$ is the ${\rm mod}~2$ reduction of the first Chern class $c_1(\xi)$ of some complex line bundle $\xi$ over $M$, then the product $F_{SO(n)}\times S(\xi)$ of the frame bundle and the circle bundle of $\xi$ admits an equivariant double covering $P_{Spin^c}(V)$ with the structural group $Spin^c(n)$. By definition, this specifies a $\mathit{Spin^c}$-\textit{structure} on $V$ associated to $\xi$, and we may often refer to the Spin$^c$ bundle $V$ as the pair $(V, \xi)$, or more explicitly the triple $(V,\xi, c_1(\xi))$. 
Furthermore, if such Spin$^c$ structure exists on $V$, 
then the distinct Spin$^c$ structures, with the determinant line bundle $\xi$, lifting the prescribed oriented structure on $V$ are in one-to-one correspondence with the elements in $H^1(M; \Z/2)$.
An excellent introduction to the structural and index theoretical aspects of Spin$^c$ structures can be found in Appendix D of the famous book of Lawson-Michelsohn \cite{Lawson89}.

In this paper, we study String$^c$ structures, which can be viewed as higher versions of Spin$^c$ structures and the complex analogue of String structures. There are interesting investigations of generalised String structures in the literature. For instance, in \cite{CHZ11} Chen-Han-Zhang studied two particular String$^c$ structures from geometric point of view, while in \cite{SSS12} Sati-Schreiber-Stasheff studied twisted String structures with physical applications.
Here, we study String$^c$ structures from the perspectives of algebraic topology, including their definitions and geometric explanations, the explicit construction of String$^c$ groups as well as the obstructions to and classification of String$^c$ structures. Parallel to the Witten genus for String manifolds, we also construct generalized Witten genera corresponding to String$^c$ conditions of various levels and obtain their vanishing theorem under the String$^c$ conditions and  give some applications.

\subsection{String$^c$ structures} As in the String case, String$^c$ structures can be understood from the perspective of the Whitehead tower. This is studied in Section \ref{strongstringcsec}. We find that one of the significant differences for this complex situation is that there are infinitely many levels of String$^c$-structures indexed by the infinite cyclic group $\mathbb{Z}$. Indeed, by analysing twisted embeddings of Spin$^c$ groups into large Spin groups, we can define a $\mathbb{Z}$-family of topological groups $String^c_k(n)$. In particular, when $k<0$, $String^c_k(n)$ is a group extension of $Spin^c(n)$ by a suitable group model of $K(\mathbb{Z},2)$ (Section \ref{strongstringcsec}). In the famous paper \cite{ST04} of Stolz-Teichner, they showed a model of $String$ as a group extension of $Spin$ by a projective unitary group $PU(A)$ of a von Neumann algebra $A$, a model of $K(\mathbb{Z}, 2)$. For our String$^c$ groups of level $k<0$, we have the extensions of topological groups
\begin{equation}\label{stringcext1eqintro}
\{1\}\rightarrow PU(A)\rightarrow String^c_k(n) \rightarrow Spin^c(n)\rightarrow \{1\}.
\end{equation}
Indeed every model of String group can induce a group model of String$^c$ group of negative level. In contrast, when $k\geq 0$, the topological group $String^c_k(n)$ can be only defined as a \textit{homotopy group extension} of $Spin^c(n)$, in the sense that, there exists a homotopy fibration
\begin{equation}\label{stringcext2eqintro}
K(\mathbb{Z}, 3)\simeq BK(\mathbb{Z}, 2)\rightarrow BString^c_k(n) \rightarrow BSpin^c(n).
\end{equation}
Notice that when $k<0$, (\ref{stringcext2eqintro}) can obtained by simply applying the classifying functor $B$ to the group extension (\ref{stringcext1eqintro}).

Then we use the classifying spaces $BString^c_k(n)$ to define String$^c$-structures. We call a Spin$^c$ bundle $(V,\xi)$ \textit{strong String$^c$ of level} $\mathit{2k+1}$ for some $k\in \mathbb{Z}$, if there is a lift  
\begin{gather*}
\begin{aligned}
\xymatrix{
& BString^c_k(n) \ar[d]  \\
M \ar[r]_{ g^\prime  \ \ \ \ \ \ } \ar@{-->}[ru]  & BSpin^c(n),
}
\end{aligned}
\label{stringcliftdiagintro}
\end{gather*}
where $g^\prime$ is the classifying map of $P_{Spin^c}(V)$. We will show that under our construction of the String$_k^c$ group, the obstruction to the lift is 
\begin{equation*}\label{strongstringcclassintro}
q_1(V,\xi)-kc_1(\xi)^2=\frac{p_1(V)-(2k+1)c_1(\xi)^2}{2}\in H^4(M;\mathbb{Z}),
\end{equation*}
where $q_1(V,\xi)=\frac{p_1(V)-c_1^2(\xi)}{2}$ is known as the first Spin$^c$ class of $(V,\xi)$ (\cite{Duan18}; cf. \cite{Thomas62}).

Another significant difference for this complex situation is that the determinant line bundle $\xi$ of the underlying Spin$^c$ structure plays a prominent role. Actually, if this obstruction class vanishes then the stable Spin bundle $V \oplus \xi_\mathbb{R}^{\oplus (-2k-1)}$ is String and, moreover, we will show that the distinct String$^c$ structures on $(V,\xi)$ are in one-to-one correspondence with the elements in the image of 
\[
\rho^\ast: H^3(M;\mathbb{Z})\rightarrow H^3(S(\xi);\mathbb{Z}),
\]
where $\rho: S(\xi)\to M$ is the projection onto $M$ from the circle bundle $S(\xi)$ of $\xi_\mathbb{R}$, the underlying rank 2 real vector bundle of $\xi$. With mild restrictions, $\rho^\ast$ is surjective or injective and then the String$^c$-structures are classified by the third cohomology $H^3(S(\xi);\mathbb{Z})$ or $H^3(M;\mathbb{Z})$. 

The String$^c$-structures can be also understood from the perspective of free loop spaces. This is studied explicitly in Section \ref{weakstringcsec}.
Recall that there is the canonical fibration $p: LM\rightarrow M$ defined by $p(\lambda)=\lambda(1)$ for any loop $\lambda\in LM$. In particular, any characteristic class of $M$ can be pulled back to $LM$, and we may use same notations for them by abuse of notation.
With this in mind, we call a Spin$^c$ bundle $(V,\xi)$ \textit{weak String$^c$ of level} $\mathit{2k+1}$ for some $k\in \mathbb{Z}$, if the determinant obstruction class
\[
\mu_1(V,\xi)-(2k+1)s_1(L\xi)c_1(\xi)\in H^3(LM;\mathbb{Z})
\]
vanishes, where $\mu_1(V,\xi)$ and $s_1(L\xi)$ are transgressed from $q_1(V,\xi)=\frac{p_1(V)-c_1^2(\xi)}{2}$ and $c_1(\xi)$ respectively. The weak String$^c$ condition can be also understood in terms of liftings of structural groups of looped principal bundles. Actually, when $k<0$, a Spin$^c$ bundle $V$ is weak String$^c$ of level $2k+1$ if and only if the structural group $LSpin^c(n)$ of the loop principal bundle $LP_{Spin^c}(V)$ over $LM$ can be lifted to the group $L\widehat{Spin^c_k}(n)$, defined by the central extension of $LSpin^c(n)$ by $U(1)$ of level $k$
\begin{equation}\label{stringcextL1eqintro}
\{1\}\rightarrow U(1) \rightarrow L\widehat{Spin^c_k}(n) \rightarrow LSpin^c(n)\rightarrow \{1\},
\end{equation}
where the level here arises from the way of embedding $Spin^c(n)$ into large $Spin$ group. In contrast, when $k\geq 0$ we do not have such description of weak String$^c$ structures by lifting of structural groups as nice as (\ref{stringcextL1eqintro})
. Nevertheless, similar to (\ref{stringcext2eqintro}), under this situation we have a homotopy group extension of level $k$ 
\begin{equation}\label{stringcextL2eqintro}
K(\mathbb{Z}, 2)\simeq BU(1)\rightarrow BL\widehat{Spin^c_k}(n) \rightarrow BLSpin^c(n),
\end{equation}
for each $k\geq 0$. Moreover, for any $k\in \mathbb{Z}$, $V$ admits a weak String$^c$ structure of level $2k+1$ if and only if there is a lift
\begin{gather*}
\begin{aligned}
\xymatrix{
& BL\widehat{Spin^c_k}(n) \ar[d]  \\
LM \ar[r]_{ Lg^\prime  \ \ \ \ \ \ } \ar@{-->}[ru]  & BLSpin^c(n).
}
\end{aligned}
\label{stringcliftLdiagintro}
\end{gather*}

As in the case of String structures, this description via loop spaces in general is weaker than the one via classifying spaces $BSpin^c(n)$, though in nice cases they are equivalent. For the aspect of counting structures, the distinct weak String$^c$ structures on $(V,\xi)$ are in one-to-one correspondence with the elements in the image of 
\[
(L\rho)^\ast: H^2(LM;\mathbb{Z})\rightarrow H^2(LS(\xi);\mathbb{Z}),
\]
where $L\rho: LS(\xi)\to LM$ is the looping of $\rho$.
In good situation, $(L\rho)^\ast$ can be surjective and then the weak String$^c$-structures are classified by the second cohomology $H^2(LS(\xi);\mathbb{Z})$. In particular, we see the role of the determinant line bundle in the loop spaces approach as well. 

In Section \ref{strongweakrelationsec},  to compare the two notions of String$^c$ structures, we will show that the distinct strong String$^c$-structures in the non-loop world can be transgressed to their weak counterparts in the loop world via the transgression diagram
\begin{gather*}
\begin{aligned}
\xymatrix{
H^3(M) \ar[d]^{\nu}  \ar[r]^{\rho^\ast}   & H^3(S(\xi)) \ar[d]^{\nu}\\
H^2(LM) \ar[r]^{(L\rho)^\ast}                & H^2(LS(\xi)).
}
\end{aligned}
\label{countstrongweakstringcrelationdiaintro}
\end{gather*}
Nevertheless, there are possibly strictly weaker String$^c$-structures than the strong ones. 

Since the String$^c$ structures can be understood from the String structures of the vector bundle plus several copies of the determinant line bundle or its complement, the fusive aspect of String$^c$ structrues can be carried out by using the corresponding descriptions of String structures in terms of fusive structures. Indeed, there are notions of fusion (fusive loop) Spin$^c$ structures of various negative levels on the loop space $LM$ in the sense of Waldorf \cite{Wal16} or Kottke-Melrose \cite{KM13}. Furthermore the transgression of the structures discussed above factors through the enhanced transgression isomorphism of Kottke-Melrose \cite{KM15} (see Subsection \ref{fusionspincsubsec}). Hence, one may view the fusion (fusive loop) Spin$^c$ structures as weak String$^c$ structures with additional fusive conditions (see Subsection \ref{fusionspincsubsec} for details).

It is worthwhile to notice that in \cite{SSS12} Sati-Schreiber-Stasheff studied so called {\it twisted String structures} for Spin manifolds, while our String$^c$ manifolds are Spin$^c$. Moreover, we study String$^c$ and LSpin$^c$ groups, as well as their classifying spaces from the perspective of algebraic topology.  

\subsection{Generalized Witten genera and vanishing theorem} Parallel to the Witten genus for String manifolds, we can construct generalized Witten genera $W^c_{2k+1; \vec a,  \vec b}(M)$ for String$^c$ manifolds of level $2k+1>0$ indexed by two integral vectors $\vec a,  \vec b$ under the condition (\ref{cond4m}), or (\ref{cond4m+2}) depending on the dimension of $M$. Such kind of invariants were constructed in \cite{CHZ11} for the first time. In this paper, we enrich them to correspond to String$^c$ manifolds of various levels. It is worthwhile to note that $W^c_{2k+1; \vec a,  \vec b}(M)$ are more flexible due to the freedom of the double vector-valued indices. As application, we obtain Liu's type vanishing theorem for $W^c_{2k+1; \vec a,  \vec b}(M)$ as follows. 

In the following, we always assume $G$ is a simply connected compact Lie group. If $M$ is level $2k+1$ String$^c$, then rationally $p_1(M)-(2k+1)c_1(\xi)^2=0$. Suppose $G$ acts smoothly on $M$ and the action lifts to the Spin$^c$ structure (and therefore has a lift to the determinant line bundle $\xi$). Since $G$ is simply connected, for $G$-equivariant characteristic classes, we must have 
\begin{equation}\label{positivedefintro}
 p_1(M)_G-(2k+1)c_1(\xi)^2_G=\alpha\cdot \pi^*q,
\end{equation}
for some $\alpha\in \Z$, where $\pi: M\times_{G}EG\to BG$ is the projection of the Borel fibre bundle, and $q\in H^4(BG)$ is the canonical generator corresponding to the generator $u^2\in H^4(BS^1)$ (see Section \ref{genwittengenussec} for details). 
We call {\em the $G$-action positive} on the level $2k+1$ String$^c$ manifold $M$ if $\alpha>0$.

The following is a motivating example for the positivity. 

\begin{example} On $\CC P^{2n}$, consider the stable almost complex structure $J$ such that 
\[
T\CC P^{2n}\oplus \R^2 \cong \mathcal{O}(1)\oplus \cdots \oplus \mathcal{O}(1) \oplus  \mathcal{O}(-1) \oplus   \cdots\oplus  \mathcal{O}(-1),
\]
where $ \mathcal{O}(1)$ is the canonical line bundle and $ \mathcal{O}(-1)$ is its dual and in the above decomposition, there are $n+1$ many $ \mathcal{O}(1)$ and $n$ many $\mathcal{O}(-1)$. It is clear that the determinant line bundle of the Spin$^c$ structure induced by $J$ is $\mathcal{O}(1)$ and $c_1(\CC P^{2n}, J)=x\in H^2(\mathbb{C}P^{2n};\mathbb{Z})$, the generator. As $p_1(\CC P^{2n})-(2n+1)c_1(\CC P^{2n}, J)^2=0$, we see that it is level $2n+1$ String$^c$. 

The linear action of $SU(2n+1)$ on $\CC P^{2n}$ obviously preserves $J$. Since $SU(2n+1)$ is simply connected, for $SU(2n+1)$-equivariant characteristic classes, we have 
\begin{equation}\label{cancp6eq1}
p_1(\CC P^{2n})_{SU(2n+1)}-(2n+1)c_1(\mathcal{O}(1))^2_{SU(2n+1)}=\alpha\cdot \pi^*q,
\end{equation}
where $q\in H^4(BSU(2n+1))$ is the generator corresponding to the generator $u^2\in H^4(BS^1)$. We claim that $\alpha=1$ and in particular the $SU(2n+1)$-linear action is positive.   

Actually consider the embedding of the circle group $S^1$ into $SU(2n+1)$ defined by
\begin{eqnarray*} f: S^1\to SU(2n+1), \ \ \ \ \lambda\to \left[\begin{array}{ccc}\lambda^{a_0}& &\\
                                     &\ddots &\\
                                     & & \lambda^{a_{2n}}
                                     \end{array}\right],
\end{eqnarray*} 
such that $a_0, a_1, \cdots a_{2n}\in \Z$  are distinct. Clearly $\sum_{i=0}^{2n} a_i=0$. Then $S^1$ acts on $\CC P^{2n}$ through the linear action of $SU(2n+1)$ in the following manner 
\[
 \lambda [z_0, z_1, \cdots, z_{2n}]=[\lambda^{a_0}z_0,  \lambda^{a_1}z_1, \cdots, \lambda^{a_{2n}}z_{2n}].
 \]
Since $a_i$'s are distinct from each other, we see that this action has $2n+1$ fixed points 
\[
[1, 0, \cdots, 0], [0,1, 0, \cdots, 0], \cdots, [0,0, \cdots, 1].
\]
The tangent space of the first fixed point has weights $a_1-a_0, a_2-a_0, \cdots, a_{2n}-a_0$ and $\mathcal{O}(1)$ restricted at the first fixed point has weight $a_0$. Then when restricted at the first fixed point, we have
\be
\begin{split}
&p_1(\CC P^{2n})_{S^1}-(2n+1)c_1(\mathcal{O}(1))^2_{S^1}\\
=&[(a_1-a_0)^2+(a_2-a_0)^2+\cdots+(a_{2n}-a_0)^2-(2n+1)a_0^2] \pi^\ast u^2\\
=&\left(\sum_{i=0}^{2n} a_i^2\right) \pi^\ast u^2.
\end{split}
\ee
Indeed, similar computations imply that when restricted to any fixed point we always have 
\begin{equation}\label{cancp6eq2}
p_1(\CC P^{2n})_{S^1}-(2n+1)c_1(\mathcal{O}(1))^2_{S^1}=\left(\sum_{i=0}^{2n} a_i^2\right)\pi^\ast u^2.
\end{equation}
However it is also clear that for $Bf: BS^1\to BSU(2n+1)$, we have $(Bf)^*(q)=\left(\sum_{i=0}^{2n} a_i^2\right)u^2$ and then by the naturality of equivariant characteristic classes and (\ref{cancp6eq1})
\begin{equation}\label{cancp6eq3}
p_1(\CC P^{2n})_{S^1}-(2n+1)c_1(\mathcal{O}(1))^2_{S^1}=\left(\sum_{i=0}^{2n} a_i^2\right)\cdot \alpha\cdot\pi^\ast u^2.
\end{equation}
Hence from (\ref{cancp6eq2}) and (\ref{cancp6eq3}), we see that $\alpha=1>0$. So this action is positive.
\end{example}
\begin{remark}\label{pin2}
(i) In \cite{Des99}, it has been shown that if $G$ contains a $Pin(2)$ and the induced action of $Pin(2)$ from that of $G$ on $M$ has a fixed point, then the $G$-action must be positive. This is based on the fact if $\eta $ is a $Pin(2)$-equivariant line bundle, then the restriction of the $Pin(2)$-representation to $S^1$ is trivial on the fibre of $\eta$ over the $Pin(2)$-fixed points. \newline
(ii) In our example, the $S^1$ action on $\mathcal{O}(1)$ has nonzero weights on each fixed point. However in \cite{Des99}, it has been shown that $Pin(2)$ actions on $\CC P^{2n}$ must have fixed points. This shows that the circle $Im(f)$ is not contained in any $Pin(2)$ subgroup of $SU(2n+1)$. Actually, the positivity seen from this circle action does not come from a $Pin(2)$ action. \end{remark}

We have the following Liu's type vanishing theorem.
\begin{theorem}[\protect Theorem \ref{Thm:action}]\label{Main}
Let $M$ be a connected compact level $2k+1$ String$^c$ manifold with $2k+1>0$. If $M$ admits an effective action of $G$ that lifts to the underlying Spin$^c$ structure and is positive, then $W^c_{2k+1; \vec a,  \vec b}(M)=0$ for all the  $(\vec a,  \vec b)$ satisfying condition \ref{cond4m} or \ref{cond4m+2}.
\end{theorem}

Theorem \ref{Main} can be applied to stable almost complex manifolds. Suppose $(M, J)$ is a compact stable almost complex manifold. Then $M$ has a canonical Spin$^c$ structure determined by $J$. If $G$ acts smoothly on $M$ and preserves the stable almost complex structure $J$, then the action of $G$ can be lifted to the Spin$^c$ structure and the determinant line bundle $\xi$. Applying the above vanishing theorem, we have
\begin{theorem}\label{Mainstableacs} 
Let $(M, J)$ be a connected compact stable almost complex manifold, which is level $2k+1$ String$^c$, i.e., $p_1(TM)=(2k+1)c_1^2(J)$ and suppose $2k+1>0$. If $M$ admits a positive effective action of $G$ preserving $J$, then $W^{c}_{2k+1; \vec a,  \vec b}(M)=0$  for all the  $(\vec a,  \vec b)$ satisfying condition \ref{cond4m} or \ref{cond4m+2}.
\end{theorem}
 
This gives an interesting obstruction to simply connected compact Lie group action preserving stable almost complex structure (Theorem \ref{genexamplethmsec} is stated and proved in a more general setting).

\begin{corollary}[\protect Theorem \ref{genexamplethmsec}]\label{genexamplethmintro} Let $(M^{2n}, J)$ be a compact $2n$-dimensional stable almost complex manifold. Suppose $J$ gives a String$^c$ structrue of level $2k+1$, i.e., $p_1(TM)=(2k+1)c_1^2(J)$. Then if 
\begin{itemize}
\item $2k-n\geq 18$, and
\item $c_1^{n}(J)\neq 0$ rationally, 
\end{itemize}
then $M$ does not admit positive effective action of a simply connected compact Lie group preserving $J$. 
\end{corollary}

The vanishing result and the proof  in Theorem \ref{genexamplethmsec} can also applied to study Lie group actions on homotopy complex projective spaces. The Petrie's conjecture \cite{Petrie72} states that if $S^1$ acts smoothly and non-trivially on the homotopy projective space $X^{2n}$, then the total Pontryagin class $p(X^{2n})$ of $X^{2n}$ must agree with that of the standard $\mathbb{C}P^n$. The conjecture was proved particularly for $X^{2n}$ with $n\leq 4$. Furthermore, Hatorri \cite{Hat78} proved the conjecture when $X^{2n}$ admits an $S^1$ invariant stable almost complex structure with $c_1=(n+1)x$, where $x\in H^2(X, \Z)$ is the generator. He also showed that when $c_1=k_0x$ with $|k_0|>n+1$, $X^{2n}$ admits no $S^1$ action preserving $J$. For the other variations of Petrie's conjecture, Dessai and Wilking \cite{DW04} proved that the total Pontryagin class $p(X^{2n})$ is standard if $X^{2n}$ admits an effective torus action of large rank. On the other hand, by applying his Spin$^c$ rigidity theorem as the main tool, Dessai \cite{Des99} proved the following result for $S^3$ actions on homotopy complex projective spaces

\begin{theorem}[\protect Dessai] \label{homotopyprojintro} 
Let $X$ be a closed smooth manifold homotopic to $\CC P^{2n}$. If $p_1(X)>(2n+1)x^2$, then $X$ does not support a nontrivial $S^3$ action.
\end{theorem}

In Section \ref{appsubsec}, we will use the vanishing result in Theorem \ref{genexamplethmsec} to give a proof of the above theorem.  We would like to point out that our vanishing theorem for the generalized Witten genus is different from the vanishing theorem in \cite{Des99}. Actually the modular invariants we construct in this paper are level 1 modular forms, while the modular invariants in \cite{Des99} are level 2 modular forms. 

\subsection{Organization of the paper}The paper is organized as follows. In Section \ref{cohomologySpincsec}, we first introduce the basic information around Spin$^c$ groups including cohomology of related spaces in low dimensions, and then compute the free suspension (transgression) of $BSpin^c$ which is the key to link the strong and weak String$^c$ structures. In Section \ref{strongstringcsec} and Section \ref{weakstringcsec}, we establish the basic theory of String$^c$ structures in the strong and weak sense respectively, including their definitions, the construction of String$^c$ groups, the geometric explanations and their structural theorems. We then study their relations in Section \ref{strongweakrelationsec} with discussions on fusive Spin$^c$ structures on loop space. In Section \ref{genwittengenussec}, we construct generalized Witten genus $W^c_{2k+1; \vec a,  \vec b}(M)$ for String$^c$ manifolds of level $2k+1$ and prove Liu's type vanishing theorem for them with some applications. We add four appendices for reference. Appendix \hyperref[AppendixA]{A}, \hyperref[AppendixB]{B}, and \hyperref[AppendixC]{C}, are devoted to various homotopy techniques used in this paper including fibration diagram techniques, cohomology suspension and transgression, and the Blakers-Massey type theorems. These materials, though some of which may be not included in standard textbooks of algebraic topology, are well known to homotopy theorists. We add them here mainly for the readers and experts in other fields, especially for geometers and mathematical physicists. The final section, Appendix \hyperref[AppendixD]{D}, provides necessary number theoretical preliminaries for defining and computing the generalized Witten genera in Section \ref{genwittengenussec}.

$\, $

\noindent{\em Conventions:}
\begin{itemize}
\item We always use $\simeq$ to denote homotopy equivalence;
\item In this paper, the manifold $M$ under consideration is always assumed to be connected;
\item Throughout the paper, $H^\ast(X)$ is used to denote the singular cohomology $H^\ast(X;\mathbb{Z})$;
\item In order to keep the consistency with our terminologies of strong and weak String$^c$, we may use the term \textit{strong String} to mean the usual String, i.e., $\frac{p_1}{2}=0$, while use the term \textit{weak String} to mean the Spin structure on loop manifolds in the sense of Waldorf \cite{Wal16}, which was also known as \textit{String structure on loop manifold} by the earlier work of Killingback \cite{K87} and McLaughlin \cite{Mclau92}.
\item In this paper, the notations for characteristic classes of vector bundles follow the usual conventions except for those of various universal bundles in Section \ref{cohomologySpincsec}, where the subscript $i$ of a universal class $x_i$ represents its cohomological degree.
\end{itemize}

$\, $

\noindent{\em Acknowledgements.}
Ruizhi Huang was supported by Postdoctoral International Exchange Program for Incoming Postdoctoral Students under Chinese Postdoctoral Council and Chinese Postdoctoral Science Foundation.
He was also supported in part by Chinese Postdoctoral Science Foundation (Grant nos. 2018M631605 and 2019T120145), and National Natural Science Foundation of China (Grant no. 11801544). He would like to thank Professor Richard Melrose for valuable discussions on fusion (fusive loop) Spin structures on free loop manifolds and bundle (bi)-gerbes for the central extensions of structural groups of principal bundles. He is also indebted to Professor Haynes Miller for the reference \cite{MV15}, and to Professor Wilderich Tuschmann for the reference \cite{DW04}.

Fei Han was partially supported by the grant AcRF R-146-000-263-114 from National University of Singapore. He is indebted to Professor Weiping Zhang and Professor Kefeng Liu for encouragement and support. He also thanks Dr. Qingtao Chen, Professor Jie Wu, Professor Siye Wu and Professor Chenchang Zhu for helpful discussion.

Haibao Duan was partially supported by National Natural Science Foundation of China (Grant nos. 11131008 and 11661131004).

The authors want to thank the anonymous referee most warmly for careful reading of our manuscript and numerous suggestions that have improved the exposition of this paper.

\numberwithin{equation}{section}
\numberwithin{theorem}{section}

\section{Some aspects of algebraic topology around Spin$^c$}\label{cohomologySpincsec}
\noindent Throughout this section, we may use same notation for both a map and its homotopy class, and add subscripts to cohomology classes to indicate their degrees unless otherwise stated. For our purpose, we only need cohomology of spaces under consideration up to dimension $4$, and the cohomology $H^i(-;\mathbb{Z})$ here should be understood as reduced cohomology with one $\mathbb{Z}$-summand omitted in $H^{0}$.
\subsection{$Spin^c(n)$ and $BSpin^c(n)$}
By definition, the topological group $Spin^c(n)$ is given by 
\begin{equation}\label{spincdef}
Spin^c(n)=(Spin(n)\times S^1)/\{\pm 1\},
\end{equation}
where $Spin(n)\cap S^1=\{\pm 1\}$.
Alternatively, it is the central extension of $SO(n)$ by the circle group $S^1$
\begin{equation}\label{spincext1}
\{1\}\rightarrow S^1\stackrel{\iota}{\rightarrow} Spin^c(n)\stackrel{p}{\rightarrow} SO(n)\rightarrow \{1\}.
\end{equation}
From (\ref{spincdef}) we have a principal bundle 
\begin{equation}\label{spinbundle}
Spin(n)\stackrel{i}{\rightarrow} Spin^c(n)\stackrel{\pi}{\rightarrow}S^1,
\end{equation}
where $i(x)=[x, 1]$ and $\pi([x, z])=z^2$ for any $(x, z)\in Spin(n)\times S^1$.
It is then easy to see that $\pi_{1}(Spin^c(n))\cong \mathbb{Z}$, the generator $s_1$ of which serves as a right homotopy inverse of $\pi$. Hence the composition map 
\[
Spin(n)\times S^1\stackrel{i\times s_1}{\rightarrow}Spin^c(n)\times Spin^c(n)\stackrel{\mu}{\rightarrow} Spin^c(n)
\]
is a weak equivalence, that is, induces isomorphisms of homotopy groups, where $\mu$ is the group multiplication of $Spin^c(n)$. Then by the Whitehead theorem it follows that (\ref{spinbundle}) splits as spaces
\begin{equation}\label{ext2split}
Spin^c(n)\simeq Spin(n)\times S^1.
\end{equation}
Since $H^{i\leq 4}(Spin(n);\mathbb{Z})\cong \mathbb{Z}\{\mu_3\}$ with the degree $|\mu_3|=3$,
we have
\begin{equation}\label{H4spinc}
H^{i\leq 4}(Spin^c(n)))\cong \mathbb{Z}\{s_1\}\oplus \mathbb{Z}\{\mu_3\}\oplus \mathbb{Z}\{s_1\mu_3\},
\end{equation}
where $xy$ denotes the cup product of $x$ and $y$. 

For classifying spaces, it is well known that with the help of Serre spectral sequence, the cohomological transgression (see Appendix \hyperref[AppendixB]{B}) $\tau$ connects the cohomology of $BSpin(n)$ with that of $Spin(n)$.
In particular,
\[
\tau: H^3(Spin(n))\rightarrow H^4(BSpin(n))
\]
is an isomorphism such that $\tau(\mu_3)=q_4$ is a typical generator of $H^4(BSpin(n))$.
Similarly, from (\ref{H4spinc}) it is easy to show that 
\begin{equation}\label{H4Bspinc}
H^{\leq 4}(BSpin^c(n))\cong \mathbb{Z}\{t_2\}\oplus \mathbb{Z}\{q_4\}\oplus \mathbb{Z}\{t_2^2\},
\end{equation}
such that $\tau(s_1)=t_2$.

\subsection{$LSpin^c(n)$ and $BLSpin^c(n)$} 
For any pointed space $X$, we have the canonical fibration
\begin{equation}\label{generalfreefib}
\Omega X\stackrel{i}{\rightarrow} LX\stackrel{p}{\rightarrow} X,
\end{equation}
where $LX={\rm map} (S^1, X)$ is the \textit{free loop space} of $X$, and $p(\lambda)=\lambda(1)$. It is clear that there is a cross section $s: X\rightarrow LX$ defined by constant loops such that $p\circ s={\rm id}_{X}$. It follows that whenever $X$ is an $H$-space, we have 
\begin{equation}\label{LHsplit}
LX\simeq \Omega X \times X
\end{equation}
as spaces, while $LX$ inherits an $H$-structure naturally from that of $X$ by point-wise multiplications. 
When $X=G$ is a topological group, $LG$ is the so-called \textit{loop group}, and 
\begin{equation}\label{LGsplit}
LG\cong \Omega G \times G.
\end{equation}
Moreover, if $G$ ($X$) is commutative (homotopically commutative), then $LG$ ($LX$) splits as groups ($H$-spaces) in (\ref{LGsplit}) ((\ref{LHsplit})). 

The classifying space of $LG$ satisfies 
\begin{equation}\label{LBG}
BLG\simeq LBG,
\end{equation}
and we have a fibration 
\begin{equation}\label{LBGfib}
G\rightarrow BLG \stackrel{p}{\rightarrow} BG,
\end{equation}
which is fibrewise homotopy equivalent to the Borel fibration
\begin{equation}\label{borelfib}
G\rightarrow G\times_G EG\rightarrow BG
\end{equation}
induced by the adjoint action of $G$.

We are now interested in $LSpin^c(n)$. First by applying the free loop functor to (\ref{spinbundle}) we obtain the fibration
\begin{equation}\label{Lspincext2}
LSpin(n)\rightarrow LSpin^c(n)\rightarrow LS^1, 
\end{equation}
where $LS^1\cong \Omega S^1\times S^1\simeq \mathbb{Z}\times S^1$ as groups. Since there is a one-to-one correspondence between the components of $\Omega Spin^c(n)$ and of $\Omega S^1$ ($\pi_{0}(\Omega Spin^c(n))\cong \pi_{0}(\Omega S^1)\cong \mathbb{Z}$), we see that 
\begin{equation}\label{kcomspin}
\Omega Spin(n)\simeq \Omega _kSpin^c(n),
\end{equation}
where $\Omega _kSpin^c(n)$ denotes the $k$-th component of $\Omega Spin^c(n)$ indexed by $k\in \mathbb{Z}\cong  \pi_{0}(\Omega Spin^c(n))$. It should be noticed that $\Omega _{0}Spin^c(n)$ is a normal subgroup of $\Omega Spin^c(n)$, and the splitting (\ref{kcomspin}) is an $A_\infty$-equivalence in this case (that is, a group isomorphism up to homotopy). Hence the $k$-th component of $LSpin^c(n)$
\begin{eqnarray*}
L_kSpin^c(n)
&\simeq& \Omega_{k} Spin^c(n) \times Spin^c(n)\\
&\simeq& \Omega Spin(n)\times Spin^c(n)\\
&\simeq& \Omega Spin(n)\times Spin(n)\times S^1,
\end{eqnarray*}
and 
\begin{equation}\label{holspinc}
H^\ast(L_kSpin^c(n))\cong H^\ast(L Spin(n)) \otimes H^\ast(S^1).
\end{equation}
In particular, 
\begin{equation}\label{H4Lkspinc}
H^{i\leq 4}(L_kSpin^c(n)))\cong \mathbb{Z}_{\leq 4}[s_1, x_2, \mu_3],
\end{equation}
where $\mathbb{Z}_{i\leq m}[-]$ denotes graded truncated polynomial ring consisting of elements of degree not greater than $m$, the generator $x_2\in H^2(\Omega Spin(n))$ satisfies $\tau(x_2)=\mu_3$.

$L_0Spin^c(n)$ is also a normal subgroup of $LSpin^c(n)$, and we have the group extension 
\[\{1\}\rightarrow L_0Spin^c(n)\rightarrow LSpin^c(n)\rightarrow \mathbb{Z}\rightarrow \{1\}.\]
Then we see that $BL_0Spin^c(n)$ is the universal covering of $BLSpin^c(n)$
\begin{equation}\label{1covblspinc}
\mathbb{Z}\rightarrow BL_0Spin^c(n)\rightarrow BLSpin^c(n).
\end{equation}
Moreover, using (\ref{Lspincext2}) we have 
\begin{equation}\label{L0spincext}
LSpin(n)\rightarrow L_0Spin^c(n)\rightarrow S^1, 
\end{equation}
which implies that $BLSpin(n)$ is the $2$-connected cover of $BL_0Spin^c(n)$, and then of $BLSpin^c(n)$
\begin{equation}\label{2covblspinc}
S^1\rightarrow BLSpin(n)\rightarrow BL_0Spin^c(n).
\end{equation}
In conclusion, we have the first two stages of the Whitehead tower of $BLSpin^c(n)$
\begin{equation}\label{whitetower}
\cdots\rightarrow BLSpin(n)\rightarrow BL_0Spin^c(n)\rightarrow BLSpin^c(n).
\end{equation}

Now the cohomology of $BL_0Spin^c(n)$ can be computed via the Serre spectral sequence of (\ref{2covblspinc}), while the cohomology of $BLSpin(n)$ and $BLSpin^c(n)$ can be calculated via that of the loop space fibration (\ref{LBGfib}). Here we need to use the fact that $p^\ast: H^\ast(BG)\rightarrow H^\ast(BLG)$ is an injection due to the existence of cross section of (\ref{LBGfib}), which allows us to handle the $E_2$-terms in low degrees easily. We summarise the results in the next subsection.

\subsection{Cohomology in low dimensions}\label{cohomcompusubsec}
Table \ref{table1} summarises the cohomology of dimensions up to $4$ for groups and their classifying spaces around Spin$^c$ based on the discussion in the last two subsections.

\begin{table}[!htbp]
\caption{Cohomology in low dimensions}
\begin{tabular}{c|p{1.4cm}lp{2.7cm}lp{3cm}lp{3cm}}
$H^{i=?}(-)$                 & 1        & 2        & 3        & 4        \\ \hline
$Spin(n)$                     & $0$          &     $0$      &     $\mathbb{Z}\{\mu_3\}$      &   $0$        \\ \hline
$\Omega Spin(n)$       &  0         &      $\mathbb{Z}\{x_2\} $     &      0     &      $\mathbb{Z}\{x_2^2\}$     \\ \hline
$LSpin(n)$                   &     0      &    $\mathbb{Z}\{x_2\}$       &   $\mathbb{Z}\{\mu_3\}$        &    $\mathbb{Z}\{x_2^2\}$       \\ \hline
$Spin^c(n)$                  &     $\mathbb{Z}\{s_1\}$      &    0       &   $\mathbb{Z}\{\mu_3\}$        &     $\mathbb{Z}\{s_1\mu_3\}$      \\ \hline
$L_kSpin^c(n)$               &      $\mathbb{Z}\{s_1\}$     &     $\mathbb{Z}\{x_2\}$      &     $\mathbb{Z}\{s_1x_2\}\oplus \mathbb{Z}\{\mu_3\}$      &     $\mathbb{Z}\{s_1\mu_3\}\oplus \mathbb{Z}\{x_2^2\}$      \\ \hline
$BSpin(n)$                   &      $0$     &     $0$      &      $0$     &    $\mathbb{Z}\{q_4\}$       \\ \hline
$BSpin^c(n)$                &      $0$     &    $\mathbb{Z}\{t_2\}$       &    $0$       &    $\mathbb{Z}\{t_2^2\}\oplus \mathbb{Z}\{q_4\}$       \\ \hline
$BLSpin(n)$                  &      $0$     &      $0$     &     $\mathbb{Z}\{\mu_3\}$      &     $\mathbb{Z}\{q_4\}$      \\ \hline
$BL_0Spin^c(n)$           &       0    &       $\mathbb{Z}\{t_2\}$    &    $\mathbb{Z}\{\mu_3\}$       &     $\mathbb{Z}\{t_2^2\}\oplus \mathbb{Z}\{q_4\}$      \\ \hline
$BLSpin^c(n)$               &        $\mathbb{Z}\{s_1\}$   &   $\mathbb{Z}\{t_2\}$        &    $\mathbb{Z}\{s_1t_2\}\oplus \mathbb{Z}\{\mu_3\}$       &      $\mathbb{Z}\{t_2^2\}\oplus \mathbb{Z}\{s_1\mu_3\}\oplus \mathbb{Z}\{q_4\}$     \\ \hline
\end{tabular}
\label{table1}
\end{table}

In the table, we indicate the generators of each group by abuse of notation, which indeed show their connections through computations and ring structures, and any two generators corresponding to each other via some map are denoted by same letter. For later use, let us also recall that we have nontrivial transgressions (for the discussions on transgressions, see Appendix \hyperref[AppendixB]{B})
\begin{equation}\label{transpin}
\tau(s_1)=t_2, \ \ \ \  \tau(x_2)=\mu_3, \ \ \ \  \tau(\mu_3)=q_4.
\end{equation}

Notice that in Table \ref{table1}, we do not consider $SO(n)$ and its relatives. Indeed, there are relations among the generators of classifying spaces. Recall that 
\begin{equation}\label{cohobso}
H^\ast(BSO(n);\mathbb{Q})\cong \mathbb{Q}[p_1,p_2,\ldots],
\end{equation}
where $p_i$ is the $i$-th universal Pontryagin class. Then by abuse of notation we have the following relations of universal characteristic classes (\cite{Duan18})
\begin{equation}\label{h4relation}
2q_4=p_1\in H^\ast(BSpin(n);\mathbb{Z}), \ \ \ \  2q_4+t_2^2=p_1\in H^\ast(BSpin^c(n);\mathbb{Z}). 
\end{equation}

\subsection{Evaluation map and free suspension}\label{suspsec}
\noindent Let $X$ be a pointed space. 
We define the free evaluation map
\begin{equation}\label{freeevalu}
{\rm ev}: S^1\times LX \rightarrow X
\end{equation}
by ${\rm ev}((t, \lambda))=\lambda(1)$. 
The \textit{free suspension}
\begin{equation}\label{freesuspen}
\nu: H^{n+1}(X) \rightarrow H^{n} (LX)
\end{equation}
is then determined by the formula ${\rm ev}^\ast(x)= 1\otimes p^\ast(x)+ s_1\otimes \nu(x)$ for any $x\in H^{n+1}(X)$ (note that it is usually called transgression by geometers, but we prefer the term free suspension here since it is naturally related to the cohomology suspension, and also the term transgression is already used for a particular type of differential in spectral sequences which is somehow the partial inverse of the cohomology suspension; see the discussions in the next paragraph and Appendix \hyperref[AppendixB]{B} for details).

It is not hard to check that the free suspension satisfies the following properties (the map $i$ and $p$ are defined in (\ref{generalfreefib}); see Section $3$ of \cite{Kuri99} or Section $2$ of \cite{KK10}):
\begin{itemize}  
\item[(1)] $i^\ast\circ \nu =\sigma^\ast: H^{n+1}(X) \rightarrow H^{n} (\Omega X)$;
\item[(2)] $\nu (xy)= \nu(x)p^\ast(y)+(-1)^{|x|}\nu(y)p^\ast(x)$, for any $x$ and $y\in H^{n+1}(X)$,
\end{itemize}
where $\sigma^\ast$ is the classic cohomology suspension (for details see Appendix \hyperref[AppendixB]{B}).
The Property (2) means that $\nu$ is a module derivation under $p^\ast$ (but since $p^\ast$ is always injective, we may omit it and simply write $\nu(x)y$ for $\nu(x)p^\ast(y)$, etc). It is helpful to mention that the transgression $\tau$ is a partial inverse of $\sigma^\ast$, and in good cases they are isomorphisms (again refer to Appendix \hyperref[AppendixB]{B}). In particular, for the transgressions in (\ref{transpin}) we have 
\begin{equation}\label{suspenspin2}
\sigma^\ast(t_2)=s_1, \ \ \ \  \sigma^\ast(\mu_3)=x_2, \ \ \ \  \sigma^\ast(q_4)=\mu_3.
\end{equation}

Let us now study the free suspension for $X=BSpin^c(n)$. We then form a commutative diagram of evaluation maps
\begin{gather}
\begin{aligned}
\xymatrix{
S^1\times LBSpin(n) \ar[rr]^{{\rm ev}}  \ar[d] &&  BSpin(n) \ar[d]\\
S^1\times LBSpin^c(n) \ar[rr]^{{\rm ev}}   \ar[d] &&  BSpin^c(n) \ar[d]\\
S^1\times LBS^1 \ar[rr]^{{\rm ev}}   && BS^1,
}
\end{aligned}
\label{evalucondiag}
\end{gather}
which implies the diagram 
\begin{gather}
\begin{aligned}
\xymatrix{
H^4(BSpin(n)) \ar[r]^{\nu}              & H^3(BLSpin(n))  \\
H^4(BSpin^c(n)) \ar[r]^{\nu}  \ar[u] & H^3(BLSpin^c(n)) \ar[u]\\
H^4(BS^1) \ar[r]^{\nu}  \ar[u]          & H^3(BLS^1) \ar[u] 
}
\end{aligned}
\label{evalucondiag2}
\end{gather}
commutes. 
The morphisms $\nu$ for $BSpin(n)$ and $BS^1$ in Diagram \ref{evalucondiag2} are easy. Indeed, since 
\[
i^\ast\circ \nu (t_2)=\sigma^\ast(t_2)=s_1,
\]
and $i^\ast: H^1(LBS^1)\rightarrow H^1(\Omega BS^1)$ is an isomorphism, we see that 
\[
\nu(t_2)=s_1.
\]
Similarly, since 
\[
i^\ast\circ \nu(q_4)=\sigma^\ast(q_4)=\mu_3,
\]
and $i^\ast: H^3(LBSpin(n))\rightarrow H^3(\Omega BSpin(n))$ is an isomorphism, we see that 
\[
\nu(q_4)=\mu_3.
\]

\begin{lemma}\label{suspenvalue}
$\nu: H^4(BSpin^c(n);\mathbb{Z})\rightarrow H^3(BLSpin^c(n);\mathbb{Z})$ satisfies
\[
\nu(q_4)=\mu_3-s_1t_2, \ \  \nu(t_2^2)=2s_1t_2,
\]
while the $i$-th component of the cohomology suspension 
\[
\sigma^\ast_i: H^3(BLSpin^c(n);\mathbb{Z}) \rightarrow H^2(L_iSpin^c(n);\mathbb{Z})
\]
satisfies
\[
\sigma^\ast_i(\mu_3)=x_2, \ \ \sigma^\ast_i(s_1t_2)=0
\]
for each $i\in \mathbb{Z}\cong \pi_1(LSpin^c(n))$.
\end{lemma}
\begin{proof}
The computations of the value of $\sigma_i^\ast$ are easy and will be omitted here. For the free suspension,
based on the previous calculations we have 
\begin{equation}\label{interkeyevalu}
\nu(t_2^2)=2s_1t_2, \ \ \ \nu(q_4)=\mu_3+\lambda s_1t_2,
\end{equation}
for some $\lambda\in \mathbb{Z}$ by Property $(2)$ of $\nu$ and the commutativity of Diagram \ref{evalucondiag2}. 
In order to get the exact value of $\lambda$, we consider the homotopy commutative diagram of fibrations
\begin{gather}
\begin{aligned}
\xymatrix{
\ast \ar[r] \ar[d]  & BSpin(n) \ar@{=}[r] \ar[d]^{Bi}  & BSpin(n) \ar[d]^{Bp}  \\
K(\mathbb{Z}, 2)  \ar@{=}[d] \ar[r]^{B\iota\ \ }  & BSpin^c(n) \ar[d]^{B\pi}  \ar[r]^{Bp}  & BSO(n) \ar[d]^{\omega_2} \\
K(\mathbb{Z},2 )  \ar[r]^{2\cdot \ \ }   &K(\mathbb{Z},2) \ar[r]^{\rho_2} & K(\mathbb{Z}/2, 2), 
}
\end{aligned}
\label{defspincdiag}
\end{gather}
where $2\cdot$ is a square map of $H$-spaces. By applying the functor $L$ to the lower part of the Diagram \ref{defspincdiag} and composing with suspension, we can form a commutative diagram 
\begin{gather}
\begin{aligned}
\xymatrix{
H^4(BS^1) \ar[d]^{\nu}  & H^4(BSpin^c(n))  \ar[l]_{(B\iota)^\ast}  \ar[d]^{\nu}   &  H^4(BSO(n)) \ar[d]^{\nu} \ar[l]_{ ~~(Bp)^\ast}\\
H^3(BLS^1) \ar@{=}[d] & H^3(BLSpin^c(n))    \ar[l]_{(BL\iota)^\ast}             &  H^3(BLSO(n))   \ar[l]_{ ~~(BLp)^\ast}              \\
H^3(BLS^1)                   & H^3(BLS^1)  \ar[l]_{(BL2\cdot)^\ast}   \ar[u]_{(BL\pi)^\ast}                        & H^3(BLK(\mathbb{Z}/2, 1)) \ar[l]_{(BL\rho_2)^\ast} \ar[u],
}
\end{aligned}
\label{lambdacaldiag}
\end{gather}
where $(BL2\cdot)^\ast=4\cdot$.
Now we need to calculate the two sides of the following equality:
\begin{equation}\label{leftrightlambda}
((BL\iota)^\ast \circ\nu)(q_4)=(\nu\circ(B\iota)^\ast) (q_4).
\end{equation}
For the left hand side of (\ref{leftrightlambda}), (\ref{interkeyevalu}) implies that
\[
((BL\iota)^\ast \circ\nu)(q_4)=(BL\iota)^\ast(\mu_3)+\lambda (BL\iota)^\ast(s_1t_2).
\]
Recall that $H^3(SO(n))\cong \mathbb{Z}\{e_3\}$ such that $\tau(e_3)=p_1$, and $H^3(Spin^c(n))\cong H^3(Spin(n))\cong \mathbb{Z}\{\mu_3\}$. Then since $\tau(\mu_3)=q_4$ and $2q_4=p^\ast(p_1)\in H^4(BSpin(n))\stackrel{p^\ast}{\leftarrow} H^4(BSO(n))$, by the naturality of $\tau$ we see that $p^\ast(e_3)=2\mu_3$ and $2(BL\iota)^\ast(\mu_3)= (BL\iota)^\ast(BLp)^\ast(e_3)=0$. It follows that $(BL\iota)^\ast(\mu_3)=0$.
Also, $(BL\iota)^\ast(s_1t_2)=(BL\iota)^\ast(BL\pi)^\ast(s_1t_2)=4s_1t_2$. 
Hence 
\begin{equation}\label{halflambdaeq}t_2
((BL\iota)^\ast \circ\nu)(q_4)=4\lambda s_1t_2.
\end{equation}
For the right hand side of (\ref{leftrightlambda}), we know that $(Bp)^\ast(p_1)=2q_4+t_2^2$ by (\ref{h4relation}), and it follows that 
\[
0=(B\iota)^\ast(Bp)^\ast(p_1)=(B\iota)^\ast(2q_4+t_2^2)=2(B\iota)^\ast(q_4)+4t_2^2.
\]
Hence, $(B\iota)^\ast(q_4)=-2t_2^2$ and by (\ref{interkeyevalu})
\begin{equation}\label{halflambdaeq2}
(\nu\circ(B\iota)^\ast) (q_4)=-2\nu(t_2^2)=-4s_1t_2
\end{equation}
Combining (\ref{leftrightlambda}), (\ref{halflambdaeq}) and (\ref{halflambdaeq2}) together, we see that $\lambda=-1$.
This proves the lemma for the value of $\nu$.
\end{proof}


\section{Strong String$^c$-structures}\label{strongstringcsec}
\noindent Let $V$ be an $n$-dimensional oriented vector bundle over a connected compact oriented smooth manifold $M$. $V$ is said to have a \textit{Spin$^c$-structure} if and only if its second Stiefel-Whitney class $\omega_2(V)$ is in the image of the mod $2$ reduction homomorphism
\[
\rho_2: H^2(M;\mathbb{Z}) \rightarrow H^2(M; \mathbb{Z}/2).
\]
Specifying such a structure is then equivalent to choosing a particular class $c\in H^2(M;\mathbb{Z})$ such that $\rho_2(c)=\omega_2(V)$, which determines and is determined by a complex line bundle $\xi$ with its associated circle bundle
\begin{equation}\label{s1bundle}
S^1\rightarrow S(\xi)\rightarrow M.
\end{equation}
We may often refer to the Spin$^c$ bundle $V$ as the pair $(V, \xi)$, or more explicitly the triple $(V,\xi, c_1(\xi))$.
Let $F_{SO}(V)\rightarrow M$ be the principal orthonormal frame bundle of $V$ with fibre $SO(n)$. Then there exists a \textit{principal $Spin^c(n)$-bundle}
\begin{equation}\label{spincframe}
Spin^c(n)\stackrel{i}{\rightarrow} P_{Spin^c}(V)\stackrel{\pi}{\rightarrow} M,
\end{equation}
defined as the fibrewise double cover of $P_{SO}(V)\times S(\xi)$
with classifying map $g: M\rightarrow BSpin^c(n)$. 

\begin{definition}\label{stringcdefBG}
Let $V$ be an $n$-dimensional real Spin$^c$-vector bundle over $M$ with the complex determinant line bundle $\xi$. For any $k\in \mathbb{Z}$, $V$ is said to be {\em level $2k+1$ strong $String^c$} if the characteristic class 
\[
\frac{p_1(V)-(2k+1)c^2}{2}=0,
\]
where $p_1(V)$ is the first Pontryagin class of $V$ and $c=c_1(\xi)$ is the first Chern class of $\xi$.  

In particular, a manifold $M$ is said to be {\em level $2k+1$ strong $String^c$} if its tangent bundle $TM$ is level $2k+1$ strong String$^c$.
\end{definition}
Let us look at the universal case and define $BString^c_k(n)$ to be the homotopy fibre of the map 
\begin{equation}\label{Bstringckdef}
\frac{p_1-(2k+1)t_2^2}{2}: BSpin^c(n)\rightarrow K(\mathbb{Z}, 4),
\end{equation}
where $p_1$ and $t_2$ are the universal classes specified in Table \ref{table1} and Subsection \ref{cohomcompusubsec} (and we will use them and other universal classes without further reference in many places of the rest of the paper).
At this moment, this is just a space with specified notation. We want to construct $BString^c_k(n)$ explicitly as the classifying space of the $(2k+1)$-level strong String$^c$-structure, and then the bundle $V$ would be a \textit{strong String$^c$-bundle of level $2k+1$} if the classifying map of the associated frame $Spin^c(n)$-bundle of $V$ can be lifted to a map to $BString^c_k$ serving as the classifying map of the desired $(2k+1)$-level String$^c$-structure
\begin{gather}
\begin{aligned}
\xymatrix{
& BString^c_k(n) \ar[d]  \\
M \ar[r]_{ g  \ \ \ \ \ \ } \ar@{-->}[ru]  & BSpin^c(n).
}
\end{aligned}
\label{spincliftdiag}
\end{gather}
For this purpose, we need to show that $BString^c_k(n)$ is really a classifying space of some topological group $String^c_k(n)$ with suitable group model, which justifies our choice of notation as well.

\subsection{Construction of String$^c$ groups}\label{stringcgroupsubsec}
Let us firstly consider the case when $k<0$. The first step is to embed the group $Spin^c(n)$ into a larger Spin group $Spin(n-4k-2)$ through the pullback of groups
\begin{gather}
\begin{aligned}
\xymatrix{
Spin^c(n)  \ar[rrrr]^{\lambda_{2k+1}}   \ar[d]^{\rho}    &&&&  Spin(n-4k-2) \ar[d]^{p}  \\
SO(n)\times S^1 \ar[rr]^{{\rm id}_{SO(n)}\times \Delta_{-2k-1} \ \ \ \ \ \ }   &&SO(n)\times \underbrace{S^1\times \cdots \times S^1}_{-2k-1}  \ar@{^{(}->}[rr]^{\ \ \ \ \chi_{-2k-1}} &&  SO(n-4k-2),
}  
\end{aligned}
\label{diaspincemb}
\end{gather}
where $\rho([x, z])=(p(x), z^2)$ ($p$ is the standard projection map; see (\ref{spincext1})), $\Delta_{-2k-1}$ is the diagonal map, and
\[
\chi_{-2k-1}(A, z_1,\ldots, z_{-2k-1})={\rm diag}(A, z_1, \ldots, z_{-2k-1})
\]
is the standard embedding mapping any $(n\times n)$-matrix $A$ and $(2\times 2)$-matrix $z_i$ to be block diagonal matrix. Then we may use the group embedding of Diagram (\ref{diaspincemb}) to define the group $String^c_k(n)$ as the pullback
\begin{gather}
\begin{aligned}
\xymatrix{
String^c_k(n)   \pullbackcorner \ar[r]^{\gamma_{2k+1}\ \ \ \ \ \   }  \ar[d]_{j_k}   & String(n-4k-2)  \ar[d]^{j}   \\
Spin^c(n)    \ar[r]^{\lambda_{2k+1} \ \ \ \ \ \ }                  &  Spin(n-4k-2),
}
\end{aligned}
\label{stringck-defdiag}
\end{gather}
where $j: String(n-4k-2)\rightarrow Spin(n-4k-2)$ can be chosen as any group extension by group model of $K(\mathbb{Z}, 2)$.

So far we have defined $String^c_k(n)$ for any $k<0$, the group structure of which can be understood through that of String group. In particular, the group models of $String$ will induce group models of $String^c$. Indeed, there is a topological group model of $String$ by
Stolz and Teichner \cite{ST04} in terms of group extension by a projective unitary group $PU(A)$ as a model of $K(\mathbb{Z}, 2)$. On the other hand, Nikolaus, Sachse and Wockel \cite{NSW13} constructed an infinite-dimensional Lie group model for $String$. In either case, we obtain a real topological or smooth group $String^c_k(n)$ as a group extension of $Spin^c(n)$.

In order to get similar definitions of $String^c_k(n)$ when $k\geq 0$, we need to modify our embeddings in Diagram (\ref{diaspincemb}). Recall that the stable special orthogonal group $SO=SO(\infty)=\lim\limits_{n} SO(n)$ 
is an infinite loop space, and in particular there is an $A_\infty$ map (i.e., a group homomorphism up to homotopy)
\[
\nu: SO\rightarrow SO,
\] 
which is the homotopy inverse of the identity map (that is, represents the loop element $\Omega[-1]$ of $[-1]$ in the group $[BSO, BSO]$). Our aim is to construct the following homotopy commutative diagram twisted by $\nu$ 
\begin{gather}
\begin{aligned}
\xymatrix{
Spin^c(n)  \ar[rrrrr]^{\lambda_{2k+1}}   \ar[d]^{\rho}    &&&&&  Spin \ar[d]^{p}  \\
SO(n)\times S^1  \ar@{^{(}->}[rr]^{{\rm id}_{SO(n)}\times \chi } && SO(n)\times SO   \ar[rr]^{{\rm id}_{SO(n)}\times \nu }  &&  SO(n)\times SO \ar@{^{(}->}[r]^{\ \ \ \  j} & SO,
}  
\end{aligned}
\label{diaspincembk+}
\end{gather}
such that $\lambda_{2k+1}$ is a loop map from $Spin^c(n)$ to the stable group $Spin$; here $j$ is the standard embedding and $\chi$ is defined as the composition
\[
S^1\stackrel{\Delta_{2k+1}}{\longrightarrow}\underbrace{S^1\times \cdots \times S^1}_{2k+1}\stackrel{\chi_{2k+1}}{\longrightarrow}SO(2k+1)\longhookrightarrow SO.
\]
For this we need to work on the level of classifying spaces. Denote the composition of the bottom maps in Diagram (\ref{diaspincembk+}) by $\phi_{2k+1}$. After applying the classifying functor $B$, it is not hard to show that there exists a homotopy commutative diagram of fibrations
\begin{gather}
\begin{aligned}
\xymatrix{
BSpin^c(n) \ar[r]^{B\rho} \ar[d]^{\Lambda_{2k+1}}  & BSO(n)\times BS^1 \ar[r]^{\omega_2\times \ubar{c}_2} \ar[d]^{B\phi_{2k+1}} & K(\mathbb{Z}/2,2) \ar@{=}[d] \\
BSpin \ar[r]^{Bp}                             &        BSO \ar[r]^{\omega_2}                                      &K(\mathbb{Z}/2,2),
}
\end{aligned}
\label{diaspincembk+2}
\end{gather}
where $\ubar{c}_2:=\rho_2(t_2)$ is the mod-$2$ reduction of $t_2\in H^2(BS^1;\mathbb{Z})$, $\Lambda_{2k+1}$ is any map induced from the homotopy commutativity of the right square. Hence, we may let $\lambda_{2k+1}=\Omega\Lambda_{2k+1}$, and in particular obtain Diagram (\ref{diaspincembk+}).

Now since $Spin^c(n)$ is compact,  $\lambda_{2k+1}$ indeed maps it into some finite stage of $Spin$ as loop map, that is, for sufficient large $m$
\[
\lambda_{2k+1}: Spin^c(n)\rightarrow Spin(m+n+4k+2).
\]
Hence, as in Diagram (\ref{stringck-defdiag}) we may define $HString^c_k(n)$ as the homotopy pullback of $\lambda_{2k+1}$ along $j: String(m+n+4k+2)\rightarrow Spin(m+n+4k+2)$ when $k\geq0$, which is clearly independent of the choice of $m$. In particular, $HString^c_k(n)$ is a loop space since $\lambda_{2k+1}$ is a loop map. Finally by passing from loops to Moore loops, there exists a topological group $String^c_k(n)$ served as the group model of the loop space $HString^c_k(n)$. 

To summarise, we have constructed $String^c_k(n)$ for any $k\in\mathbb{Z}$. When $k<0$, $String^c_k(n)$ is a group extension of $Spin^c(n)$ by any suitable model of $K(\mathbb{Z},2)$, and can be embedded to a large String group. However, when $k\geq 0$, the group $String^c_k(n)$ is constructed neither as a group extension of $Spin^c(n)$ nor as a group directly related to String group. The reason is due to the lack of a self group homomorphism $\nu^\prime$ of $SO$ such that $(B\nu^\prime)^\ast (p_1)=-p_1$ for the first Pontryagin class. 
Indeed, if such a self group homomorphism exists, its effect on the cohomology of the maximal torus implies that $t_2^2\in H^4(BS^1)$ will be sent to $-t_2^2$ for some subgroup $S^1$ of $SO$. But it is clear that this can not happen.
Nevertheless, there are still relations among the groups at the homotopical level when $k\geq 0$. Informally we may say $String^c_k(n)$ is a homotopy group extension of $Spin^c(n)$ by $K(\mathbb{Z},2)$, and a homotopy subgroup of String as well. This just means that both relations are only valid in the  homotopy category. Since we only need to deal with classifying spaces and maps among them, these descriptions of the groups $String^c_k(n)$ are sufficient for our purpose in this paper. 

\subsection{Classifying spaces and counting strong String$^c$ structures}\label{classifyspsubsec}
Let us check that our constructions are the right choices for the defining obstructions of String$^c$ structures. 
Applying the classifying functor $B$ to Diagram (\ref{diaspincemb}), there is particularly an $SO(n-4k-2)$-bundle over $BSO(n)\times BS^1$ with first Pontryagin class $p_1-(2k+1)t_2^2$ presented by the bottom composition. Since by (\ref{h4relation})
\[(B\rho)^\ast(p_1-(2k+1)t_2^2)=2q_4-2kt_2^2,\]
and also $p_1=2q_4$ in $H^4(Spin(n+4k+2))$, we see that
\begin{equation}\label{Bgeoobstruction0}
(B\lambda_{2k+1})^\ast(q_4)=q_4-kt_2^2.
\end{equation}
Applying the classifying functor $B$ to Diagram (\ref{stringck-defdiag}), by (\ref{Bgeoobstruction0}) we have the commutative diagram
\begin{gather}
\begin{aligned}
\xymatrix{
BSpin^c(n) \ar[r]^{B\lambda_{2k+1} \ \ \ \ \ }  \ar[rd]_{\frac{p_1-(2k+1)t_2^2}{2}} & BSpin(n-4k-2) \ar[d]^{\frac{p_1}{2}}  \\ 
&  K(\mathbb{Z}, 4),
}
\end{aligned}
\label{Bstringckliftdia}
\end{gather}
which justifies the definition of $BString^c_k(n)$ by (\ref{Bstringckdef}) for $k<0$. The case when $k\geq0$ can be treated similarly with the facts that $(B\nu)^\ast(p_1)=-p_1$ and then $p_1$ will be pulled back to $p_1-(2k+1)t_2^2\in H^4(SO(n)\times S^1)$ along the bottom composition in Diagram (\ref{diaspincembk+}).

The process of constructing these groups also suggests geometric explanations for the String$^c$-structures. Let $\xi_{\mathbb{R}}$ be the underlying real bundle of the determinant line bundle $\xi$. 
For our Spin$^c$-bundle $V$ over $M$, let us consider the real $(n-4k-2)$-bundle $V\oplus \xi^{\oplus (-2k-1)}_{\mathbb{R}}$ when $k<0$. Then it is easy to calculate its second Stiefel-Whitney class 
\[
\omega_2(V\oplus \xi_{\mathbb{R}}^{\oplus (-2k-1)})=\omega_2(M)-(2k+1)c_1(\xi)~{\rm mod}~2=0,
\]
where $c_1(\xi)$ is the first Chern class of $\xi$.
In particular, the principal frame bundle $P_{SO}(V\oplus \xi_{\mathbb{R}}^{\oplus (-2k-1)})$ has a fibrewise two-sheeted covering $P_{Spin}^k(V, \xi)$, which is a Spin bundle
\begin{equation}\label{kspinframe}
Spin(n-4k-2)\stackrel{i}{\rightarrow}P_{Spin}^k(V,\xi)\stackrel{\pi}{\rightarrow} M.
\end{equation}
Then by Diagram (\ref{diaspincemb}) and our above calculations, there is a bundle embedding 
\begin{gather}
\begin{aligned}
\xymatrix{
Spin^c(n) \ar[r]^{i} \ar@{_{(}->}[d]^{\lambda_{2k+1}}  &  P_{Spin^c}(V) \ar@{_{(}->}[d]^{\Theta_{2k+1}} \ar[r]^{\ \ \ \pi}  & M \ar@{=}[d] \ar[r]^{g \ \ \ \ \ \ \ }   & BSpin^c(n) \ar[d]^{B\lambda_{2k+1}}\\
Spin(n-4k-2) \ar[r]^{i}    & P_{Spin}^k(V, \xi) \ar[r]^{\ \ \ \pi}  & M \ar[r]^{h\ \ \ \ \ \ \ \ \ \ \ }  & BSpin(n-4k-2).
}  
\end{aligned}
\label{spinspincrelationweak}
\end{gather}
When $k\geq 0$, we may consider the stable vector bundle $V\oplus \xi^{\oplus (-2k-1)}_{\mathbb{R}}$, and go through the argument above with Diagram (\ref{diaspincembk+}) to map the bundle $P_{Spin^c}(V)$ to $P_{Spin}^k(V, \xi)$ similarly (while $\Theta_{2k+1}$ is only a map of principal homotopy fibrations). Note that in this case the bundle $P_{Spin}^k(V, \xi)$ is of dimension $(m+n+4k+2)$. Recall that a Spin bundle $E$ admits a \textit{strong String structure} if $\frac{p_1(E)}{2}=0$.

\begin{theorem}\label{strongstringcgeothm}
Let $V$ be an $n$-dimensional Spin$^c$-vector bundle over $M$ with a complex determinant line bundle $\xi$.
$V$ admits a strong String$^c$-structure if and only if the stable Spin bundle associated to $V\oplus \xi^{\oplus (-2k-1)}$ admits a strong String structure for some $k\in \mathbb{Z}$. 

Furthermore, if $V$ is level $2k+1$ strong String$^c$, .i.e.  the obstruction class $\frac{p_1(V)-(2k+1)c_1(\xi)^2}{2}=0$, then the $(2k+1)$-level String$^c$-structures on $V$ are in one-to-one correspondence with the elements in the image of the morphism
\[
\rho^\ast: H^3(M)\rightarrow H^3(S(\xi)),
\]where $\rho: S(\xi)\to M$ is the circle bundle of $\xi$.
\end{theorem}
\begin{proof}
We may consider the diagram 
\begin{gather}
\begin{aligned}
\xymatrix{
M \ar@/_/[ddr]_g \ar@/^/@{.>}[drr]^{(2)} \ar@{.>}[dr]|-{(1)}\\
&BString^c_k(n)   \ar[r]^{B\gamma_{2k+1}  }  \ar[d]_{Bj_k}   & BString  \ar[d]^{Bj}   \\
& BSpin^c(n)    \ar[r]^{B\lambda_{2k+1}  }                  &  BSpin,
}
\end{aligned}
\label{Bstringcstringequidia}
\end{gather}
where the square is a homotopy pullback by (\ref{stringck-defdiag}) or (\ref{diaspincembk+}). Then by the universal property of homotopy pullback, the existence of a lifting at $(1)$ in the diagram is equivalent to the existence of a lifting at $(2)$. 
For the bundle $V\oplus \xi^{\oplus (-2k-1)}$, it is easy to show that its first Pontryagin class is
\[
p_1(V\oplus \xi_\mathbb{R}^{\oplus (-2k-1)})=p_1(V)-(2k+1)c_1(\xi)^2
\]
(notice that $c_1(\xi\oplus\bar{\xi})=0$ and $p_1(\xi_\mathbb{R}^{\oplus (-2k-1)})=-(2k+1)c_1(\xi)^2$).
Then by definition, the Spin bundle associated to $V\oplus \xi^{\oplus (-2k-1)}$ admits a (strong) String structure if and only if
\[
\frac{p_1(V)-(2k+1)c_1(\xi)^2}{2}=0.
\]
This proves the first claim of the theorem. 

For the second claim of the theorem, we first prove that the different String$^c$-structures on $V$ are classified by the image of 
\[
\pi^\ast: H^3(M)\rightarrow H^3(P_{Spin^c}(V)).
\]
By Diagram (\ref{spinspincrelationweak}) we can construct a commutative diagram 
\begin{gather}
\begin{aligned}
\xymatrix{
H^3(M)  \ar[r]^{\pi^\ast \ \ \ \ } \ar@{=}[d]    & H^3(P_{Spin^c}(V)) \ar[r]^{i^\ast}  & H^3(Spin^c(n)) \ar@{=}[d] \ar[r]^{ \ \ \ \   \widetilde{\delta}_k} &
H^4(M) \ar@{=}[d]  \\
H^3(M)  \ar[r]^{\pi^\ast} \ar@{=}[d]    & {\rm Im}~\Theta^\ast \ar[r]^{i^\ast \ \ \ } \ar@{^{(}->}[u] & H^3(Spin^c(n)) \ar[r]^{\ \ \ \   \widetilde{\delta}_k  } &
H^4(M) \ar@{=}[d]  \\
H^3(M)  \ar@{^{(}->}[r]^{\pi^\ast \  \  \  \ }         & H^3(P_{Spin}^k(V, \xi)) \ar[r]^{ i^\ast \ \ \ \ \ }  \ar@{->>}[u]_{\Theta^\ast}  
&  H^4(BSpin)\cong H^3(Spin) \ar[r]^{ \ \ \ \ \ \ \ \ \ \  \widetilde{\delta}} \ar[u]_{\lambda_{2k+1}^\ast}^{\cong}       & H^4(M),
}
\end{aligned}
\label{distinctstringcstrongdia}
\end{gather}
where $\widetilde{\delta}_{k}$ is defined by $\widetilde{\delta}_{k} (\mu_3)= q_1(V)-kc_1(\xi)^2$, and the third row is exact by applying the dual Blakers-Massey theorem (Theorem \ref{dualBM}) to the lower right part of Diagram (\ref{spinspincrelationweak})
\[
P_{Spin}^k(V,\xi)\stackrel{\pi}{\rightarrow} M\stackrel{h}{\rightarrow}BSpin(n-4k-2)
\]
(notice that $BSpin(n-4k-2)$ is $3$-connected and $P_{Spin}^k(V, \xi)$ is connected). Here 
\[
q_1(V)=g^\ast(q_4): H^4(M)\leftarrow H^4(BSpin^c(n))
\]
is the characteristic class defined by universal class $q_4$. It is easy to see that the second row of the diagram is exact (this gives a second proof for the first claim). 
Notice that ${\rm Ker}\Theta^\ast\subseteq H^3(M)$ and 
the first morphism $\pi^\ast$ in the second row has ${\rm Ker}\Theta^\ast$ as its kernel. Hence the distinct String$^c$ structures on $V$ are classified by 
\[
{\rm Ker}i^\ast\cong {\rm Im}\pi^\ast\cong H^3(M)/{\rm Ker}\Theta^\ast.
\]
On the other hand, there is a bundle morphism
\begin{gather}
\begin{aligned}
\xymatrix{
Spin(n) \ar[r]  \ar[d] &  P_{Spin^c}(V) \ar[r]^{\ \ \ \pi}  \ar@{=}[d]  &  S(\xi)  \ar[r]^{\tilde{g} \ \ \ \ }  \ar[d]^{\rho} & BSpin(n) \ar[d]\\
Spin^c(n) \ar[r]        &  P_{Spin^c}(V) \ar[r]^{\ \ \ \pi}                   &  M \ar[r]^{g \ \ \ \ }                     & BSpin^c(n),
}
\end{aligned}
\label{spinspinclinedia}
\end{gather}
where the existence of lifting $\tilde{g}$ is due to the vanishing of the second Stiefel-Whitney class of $S(\xi)$. This diagram then induces a commutative diagram of cohomology groups
\begin{gather}
\begin{aligned}
\xymatrix{
0=H^2(Spin(n)) \ar[r] & H^3(S(\xi)) \ar[r] ^{\pi^\ast\ \ \ } & H^3(P_{Spin^c}(V)) \ar@{=}[d]  \\
                          &   H^3(M)  \ar[r]^{\pi^\ast\ \ \ }  \ar[u]^{\rho^\ast } & H^3(P_{Spin^c}(V)), 
}
\end{aligned}
\label{spinspinclinehodia}
\end{gather}
where the first row is exact again by Theorem \ref{dualBM}.
Hence ${\rm Im}\pi^\ast\cong {\rm Im}\rho^\ast$ and the proof of the theorem is completed.
\end{proof}

There are some cases when $\rho^\ast$ are surjective or injective.
\begin{corollary}\label{rhosurjstringccor}
Let $(V,\xi)$ as in Theorem \ref{strongstringcgeothm}. Then 
\begin{itemize}
\item[(1).] if the cup product by $c_1(\xi)$
\[
\cup c_1(\xi): H^2(M;\mathbb{Z})\rightarrow H^4(M;\mathbb{Z})
\]
is injective, then $\rho^\ast$ is surjective. In particular, the strong String$^c$ structures of level $2k+1$ on $V$ are in one-to-one correspondence with elements of $H^3(S(\xi))$;
\item[(2).]
if the fundamental group $\pi_1(M)$ is a torsion group (e.g., when $M$ is simply connected), then $\rho^\ast$ is injective. In particular, the strong String$^c$ structures of level $2k+1$ on $V$ are in one-to-one correspondence with elements of $H^3(M)$.
\end{itemize}
\end{corollary}
\begin{proof}
Let us look at the Gysin sequence of the line bundle $\xi$
\[
\cdots \rightarrow H^1(M)\stackrel{\cup c_1(\xi)}{\rightarrow} H^3(M)\stackrel{\rho^\ast}{\rightarrow} H^3(S(\xi))\stackrel{\delta}{\rightarrow} H^2(M)\stackrel{\cup c_1(\xi)}{\rightarrow}H^4(M)\rightarrow \cdots.
\]
For Case $(1)$, in the exact sequence the second cup product $\cup c_1(\xi)$ is injective, which implies that $\delta$ is trivial. Hence $\rho^\ast$ is surjective, and by Theorem \ref{strongstringcgeothm} the String$^c$ structures on $V$ are classified by ${\rm Im}~\rho^\ast=H^3(S(\xi))$. 

For Case $(2)$, the condition on the fundamental group of $M$ is equivalent to that $H^1(M)=0$. Then from the Gysin sequence above we see that $\rho^\ast$ is injective, and again by Theorem \ref{strongstringcgeothm} the String$^c$ structures on $V$ are classified by ${\rm Im}~\rho^\ast\cong H^3(M)$. 
\end{proof}

\section{Weak String$^c$-structures}\label{weakstringcsec}
\noindent Motivated by the way that (weak) String structures can be studied in terms of Spin structures on loop spaces, we define String$^c$-structures in terms of Spin$^c$-structures on loop spaces, which in general is weaker than the notion of String$^c$ defined in Section \ref{strongstringcsec}.

Let $(V,\xi)$ be the Spin$^c$-bundle defined in Section \ref{strongstringcsec}. By applying free loop functor to (\ref{spincframe}), we get a principal fibre bundle 
\begin{equation}\label{Lspincframe}
LSpin^c(n)\stackrel{Li}{\rightarrow} LP_{Spin^c}(V)\stackrel{L\pi}{\rightarrow} LM
\end{equation}
classified by $Lg: LM\rightarrow BLSpin^c(n)$. In particular, we may define the \textit{$LSpin^c$ characteristic classes} of $M$ as the pullbacks of the elements of $H^\ast(BLSpin^c(n))$ in $H^\ast(LM)$ through $Lg$. In low degrees, let us denote by $s=s(L\xi)$, $c=c_1(\xi)$, $\mu_1(V)=\mu_1(V,\xi)$, $q_1(V)=q_1(V,\xi)$ and $p_1(V)$ the $LSpin^c$-classes of $LV$ corresponding to the universal classes $s_1$, $t_2$, $\mu_3$, $q_4$ and $p_1$ respectively. We then notice that $c_1(\xi)$ and $p_1(V)$ correspond to usual Euler class of $\xi$ and the first Pontryagin class of $V$ respectively via the projection $p$ in the loop fibration (\ref{generalfreefib}), which justifies our notations.

Throughout the remanning part of this section, let us assume $x_2\in H^2(LSpin^c(n))$ is always chosen from the $0$-th component $H^2(L_0Spin^c(n))$. Recall that the cohomology suspension $\sigma^\ast$ is trivial on decomposable elements and $\sigma^\ast(\mu_3)=x_2$ (Lemma \ref{suspenvalue}; also see Appendix \hyperref[AppendixB]{B}).

\begin{definition}\label{stringcdefLG}
Let $V$ be an $n$-dimensional Spin$^c$-vector bundle over a manifold $M$ with a complex determinant line bundle $\xi$. $V$ is said to be {\em level $2k+1$ weak $String^c$} if the obstruction class 
\[
\delta_k(x_2)=\mu_1(V)-(2k+1)sc
\]
vanishes, where $\delta_k$ is the composition 
\begin{equation}\label{deltadef}
H^2(LSpin^c(n);\mathbb{Z})\stackrel{s_k}{\rightarrow} H^3(BLSpin^c(n);\mathbb{Z})\stackrel{(Lg)^\ast}{\rightarrow} H^3(LM;\mathbb{Z}),
\end{equation}
and $s_k$ is a section of cohomology suspension $\sigma^\ast: H^3(BLSpin^c(n);\mathbb{Z})\rightarrow H^2(LSpin^c(n);\mathbb{Z})$ defined by $s_k(x_2)=\mu_3-(2k+1)s_1t_2$ for each integer $k\in \mathbb{Z}$.
\end{definition}
This definition of weak String$^c$-structures also has geometric explanations. By using the group and bundle embeddings (e.g., Diagram (\ref{diaspincemb}), Diagram (\ref{spinspincrelationweak})) constructed in Section \ref{strongstringcsec}, we want to construct a commutative diagram (when $k<0$)
\begin{gather}
\begin{aligned}
\xymatrix{
H^2(L_0Spin^c(n)) \ar[r]^{s_k}   & H^3(BLSpin^c(n)) \ar[r]^{\ \ \ \ \  \ (Lg)^\ast}  & H^3(LM) \ar@{=}[d] \\
H^2(LSpin(n-4k-2)) \ar[r]^{\tau} \ar[u]^{\cong}_{(L\lambda_{2k+1})^\ast} & H^3(BLSpin(n-4k-2)) \ar[u]_{(BL\lambda_{2k+1})^\ast} \ar[r]^{\ \ \ \ \ \ \ \ \ \ \ \ (Lh)^\ast}   & H^3(LM).
}
\end{aligned}
\label{Hspinspincrelationweakobs}
\end{gather}
For this purpose, firstly apply free loop functor to Diagram (\ref{diaspincemb}), and denote $\phi=\chi_{2k+1}\circ ({\rm id}_{SO(n)}\times \Delta_{-2k-1})$. Recall that (\ref{Bgeoobstruction0})
\[
(B\phi)^\ast(p_1)=p_1-(2k+1)t_2^2,   \  \ (B\lambda_{2k+1})^\ast(q_4)=q_4-kt_2^2.
\]
Then by the naturality of the free suspension $\nu$ and Lemma \ref{suspenvalue}, the homomorphism
\[
(BL\lambda_{2k+1})^\ast: H^{3}(BLSpin(n+4k+2))\rightarrow H^3(BLSpin^c(n))
\]
satisfies
\begin{equation}\label{geoobstruction}
(BL\lambda_{2k+1})^\ast(\mu_3)=\mu_3-(2k+1)s_1t_2=s_k(x_2).
\end{equation}
Similarly, by applying cohomology suspensions for the both sides of (\ref{geoobstruction}), we obtain 
\begin{equation}\label{Llambdah2}
(L\lambda_{2k+1})^\ast(x_2)=x_2.
\end{equation}
Combining (\ref{geoobstruction}), (\ref{Llambdah2}) and the fact $\tau(x_2)=\mu_3$ for the transgression homomorphism, we see that the left square of Diagram (\ref{Hspinspincrelationweakobs}) commutes. The right square of Diagram (\ref{Hspinspincrelationweakobs}) is natural by applying loop functor $L$ to Diagram (\ref{spinspincrelationweak}).

We have showed the commutativity of Diagram (\ref{Hspinspincrelationweakobs}) when $k<0$, while the case when $k\geq 0$ can be done similarly. From the diagram, we notice that the composition of the morphisms in the first row is the defined String$^c$-obstruction $\delta_k$, while the composition $\delta$ of those in the second row is the obstruction to the existence of String structure on the bundle $V\oplus \xi^{\oplus (-2k-1)}$ from the point of view of loop spaces. Indeed, by observing the Serre spectral sequence of the spin bundle (\ref{kspinframe}) after looping (or simply applying the dual Blakers-Massey Theorem), the second row of Diagram (\ref{Hspinspincrelationweakobs}) can be fitted into an exact sequence 
\begin{equation}\label{spinstringweakn+4k+2}
0\rightarrow H^2(LM)\stackrel{(L\pi)^\ast}{\rightarrow} H^2(LP_{Spin}^k(V,\xi)) \stackrel{(Li)^\ast}{\rightarrow} H^2(LSpin(n-4k-2)) \stackrel{\delta}{\rightarrow} H^3(LM). 
\end{equation}
As in \cite{Mclau92}, $x_2\in H^2(LSpin(n-4k-2))\cong [LSpin(n-4k-2), BS^1]$ corresponds to the universal central extension $L\widehat{Spin}(n-4k-2)$ of $LSpin(n-4k-2)$ by $S^1$. If $\delta(x_2)=0$ the exactness of the above sequence implies that the structural  group of $LP_{Spin}^k(V, \xi)\rightarrow LM$ can be lifted to $L\widehat{Spin}(n-4k-2)$, which, by definition, assures a \textit{weak String structure} on $P_{Spin}^k(V, \xi)$ (cf. \cite{K87, Mclau92, Wal16}). Hence, a level $2k+1$ weak String$^c$-structure on $M$ induces a (weak) String structure on a larger bundle $V\oplus \xi^{\oplus (-2k-1)}$ over $M$.

Hence from the geometric explanation, we can interpret the String$^c$-structures in terms of liftings of structural groups. Indeed, we can define $L\widehat{Spin^c_k}(n)$ for $k<0$ by the morphism of groups extensions
\begin{gather}
\begin{aligned}
\xymatrix{
\{1\} \ar[r]   &  U(1) \ar[r] \ar@{=}[d] & L\widehat{Spin^c_k}(n) \ar[r]^{q_{2k+1}} \ar[d]^{L\widehat{\lambda_{2k+1}}} & LSpin^c(n) \ar[r] \ar[d]^{L\lambda_{2k+1}} & \{1\} \\
\{1\} \ar[r]            &  U(1) \ar[r]                  & L\widehat{Spin}(n-4k-2) \ar[r]^{q}            & LSpin(n-4k-2) \ar[r] & \{1\}, 
}
\end{aligned}
\label{Lgroupextmordia}
\end{gather}
where the bottom row is the universal extension of $LSpin$, $\lambda_{2k+1}$ is the group embedding defined in Diagram \ref{diaspincemb}, and the right square is a pullback defining the homomorphisms $L\widehat{\lambda_{2k+1}}$ and $q_{2k+1}$. Recall that by (\ref{geoobstruction}), after applying the classifying functor $B$ to (\ref{Lgroupextmordia}) the universal obstruction class $\mu_3$ of weak String structures will be sent to the universal obstruction class $s_k(x_2)$ of weak String$^c$ structures via $(BL\lambda_{2k+1})^\ast$. In particular, for the looped classifying map $Lg$ of the Spin$^c$ principal bundle of $V$ (\ref{Lspincframe}), it can be lifted to $BL\widehat{Spin^c_k}(n)$ if and only if the composition $Lh=BL\lambda_{2k+1} \circ Lg$, as the classifying map of $LP_{Spin}^k(V,\xi)$, can be lifted to $L\widehat{Spin}(n-4k-2)$.

In contrast, when $k\geq 0$ since we only have a loop map fitting into Diagram \ref{diaspincembk+}, we cannot construct the morphism of group extension as in Diagram \ref{Lgroupextmordia}, but instead we can formulate a homotopy commutative diagram of fibrations
\begin{gather}
\begin{aligned}
\xymatrix{
BS^1\ar[r] \ar[r] \ar@{=}[d] & BL\widehat{Spin^c_k}(n) \ar[r]^{\mathfrak{q}_{2k+1}} 
\ar[d]^{L\widehat{\Lambda_{2k+1}}} & BLSpin^c(n) \ar[d]^{BL\lambda_{2k+1}}  \\
BS^1 \ar[r]                  & BL\widehat{Spin}(N) \ar[r]^{Bq}            & BLSpin(N), 
}
\end{aligned}
\label{Lgroupextmor+dia}
\end{gather}
where $N$ is sufficiently large, and $BL\widehat{Spin^c_k}(n)$ is just a topological space as the homotopy pullback of the right square at this moment. Nevertheless, we can justify that it can be chosen as the classifying space of some topological group $L\widehat{Spin^c_k}(n)$ analogous to the arguments of constructing the String$^c$ groups of negative levels in Subsection \ref{stringcgroupsubsec}. Indeed, first we can take the homotopy pullback of $q$ and $L\lambda_{2k+1}$ to obtain a space $L\widehat{Spin^c_k}(n)$. Moreover notice that the maps $Bq$ and $BL\lambda_{2k+1}$ induce the morphisms of the universal fibrations of involved classifying spaces respectively, we indeed have a homotopy fibration
\[
L\widehat{Spin^c_k}(n) \rightarrow \ast \rightarrow BL\widehat{Spin^c_k}(n).
\]
Hence the space $L\widehat{Spin^c_k}(n)$ can be chosen to the Moore loop space corresponding to the loop space $\Omega BL\widehat{Spin^c_k}(n)$ as we did for $String^c_k(n)$ when $k\geq0$, and then $B\Omega BL\widehat{Spin^c_k}(n)\simeq BL\widehat{Spin^c_k}(n)$. Let $q_{2k+1}=\Omega \mathfrak{q}_{2k+1}$ and $L\widehat{\lambda_{2k+1}}=\Omega L\widehat{\Lambda_{2k+1}}$. Then we can re-choose $Bq_{2k+1}$ to be $\mathfrak{q}_{2k+1}$ and $BL\widehat{\lambda_{2k+1}}$ to be $L\widehat{\Lambda_{2k+1}}$.
 To summarize, when $k\geq 0$, we have constructed a homotopy commutative diagram similar to Diagram \ref{Lgroupextmordia} where $L\widehat{Spin^c_k}(n)$ is a topological group and the maps $q_{2k+1}$ and $L\widehat{\lambda_{2k+1}}$ are loop maps.

\begin{theorem}\label{weakstringcthm}
Let $V$ be an $n$-dimensional Spin$^c$-vector bundle over a manifold $M$ with a complex determinant line bundle $\xi$. 
$V$ admits a weak String$^c$-structure if and only if the stable Spin bundle associated to $V\oplus \xi^{\oplus (-2k-1)}$ admits a weak String structure for some $k\in \mathbb{Z}$.
Furthermore, when $k<0$, a weak String$^c$-structure of level $2k+1$ on $V$ is also equivalent to, for the associated LSpin$^c$-bundle $LP_{Spin^c}(V)$ over $LM$, a structural group lifting to $L\widehat{Spin^c_k}(n)$.

Suppose the obstruction class $\mu_3(V)-(2k+1)sc=0$, then the weak String$^c$-structures of level $2k+1$ on $(V,\xi)$ are in one-to-one correspondence with the elements in the image of the morphism
\[
(L\rho)^\ast: H^2(LM)\rightarrow H^2(LS(\xi)).
\]
\end{theorem}

\begin{proof}
First notice that we have proved the first two statements in the previous discussions. 
We now prove the last statement of the theorem on counting the distinct String$^c$ structures for $k< 0$ while the proof for $k\geq0$ is similar. And the proof is similar to that of Theorem \ref{strongstringcgeothm}.
By Diagram (\ref{spinspincrelationweak}) after looping, we can construct a commutative diagram 
\begin{gather}
\begin{aligned}
\xymatrix{
H^2(LM)  \ar[r]^{(L\pi)^\ast \ \ \ \ } \ar@{=}[d]    & H^2(LP_{Spin^c}(V)) \ar[r]^{(Li)^\ast}  & H^2(LSpin^c(n)) \ar@{=}[d] \ar[r]^{ \ \ \ \   \delta_k} &
H^3(LM) \ar@{=}[d]  \\
H^2(LM)  \ar[r]^{(L\pi)^\ast} \ar@{=}[d]    & {\rm Im}~(L\Theta)^\ast \ar[r]^{(Li)^\ast \ \ \ } \ar@{^{(}->}[u] & H^2(LSpin^c(n)) \ar[r]^{\ \ \ \   \delta_k  } &
H^3(LM) \ar@{=}[d]  \\
H^2(LM)  \ar@{^{(}->}[r]^{(L\pi)^\ast \  \  \  \ }         & H^2(LP_{Spin}^k(V, \xi)) \ar[r]^{(Li)^\ast \ \ }  \ar@{->>}[u]_{(L\Theta)^\ast}  
&  H^2(LSpin(n-4k-2)) \ar[r]^{\ \ \ \ \  \  \ \ \delta} \ar[u]_{(L\lambda_{2k+1})^\ast}^{\cong}       & H^3(LM),
}
\end{aligned}
\label{distinctstringcweakdia}
\end{gather}
where the third row is exact. It is easy to see that the second row of the diagram is exact. Now notice that ${\rm Ker}(L\Theta)^\ast\subseteq H^2(LM)$ and 
the first morphism $(L\pi)^\ast$ in the second row has ${\rm Ker}(L\Theta)^\ast$ as its kernel. 
Hence the distinct String$^c$ structures on $V$ are classified by 
\[
{\rm Ker}(Li)^\ast\cong  H^2(LM)/{\rm Ker}(L\Theta)^\ast\cong {\rm Im}((L\pi)^\ast: H^2(LM)\rightarrow H^2(LP_{Spin^c}(V)).
\]
On the other hand, by considering the bundle morphism in Diagram (\ref{spinspinclinedia}) after looping, we obtain the commutative diagram of cohomology groups
\begin{gather}
\begin{aligned}
\xymatrix{
0=H^1(LSpin(n)) \ar[r] & H^2(LS(\xi)) \ar[r] ^{(L\pi)^\ast\ \ \ } & H^2(LP_{Spin^c}(V)) \ar@{=}[d]  \\
                          &   H^2(LM)  \ar[r]^{(L\pi)^\ast\ \ \ }  \ar[u]^{(L\rho)^\ast } & H^2(LP_{Spin^c}(V)), 
}
\end{aligned}
\label{Lspinspinclinedia}
\end{gather}
where the first row is exact again by Theorem \ref{dualBM}.
Then ${\rm Im}(L\pi)^\ast\cong {\rm Im}(L\rho)^\ast$ and the proof of the theorem is completed.
\end{proof}

\begin{remark}
Notice when $k\geq 0$, we cannot talk about lifting the structural groups of the looped Spin$^c$ principal bundles. However, it is true that $V$ admits a weak String$^c$ structure of level $2k+1$ if and only if there is a lift
\begin{gather}
\begin{aligned}
\xymatrix{
& BL\widehat{Spin^c_k}(n) \ar[d]  \\
LM \ar[r]_{ Lg \ \ \ \ \ \ } \ar@{-->}[ru]  & BLSpin^c(n).
}
\end{aligned}
\label{stringcliftLdiag}
\end{gather}
\end{remark}

\begin{corollary}\label{rhosurjweakstringccor}
Let $(V,\xi)$ as in Theorem \ref{weakstringcthm}. Suppose that $M$ is simply connected, and $c_1(\xi)$ is a generator element of $H^2(M)$,
then $(L\rho)^\ast: H^2(LM)\rightarrow H^2(LS(\xi))$ is surjective. In particular, the weak String$^c$ structures of level $2k+1$ on $V$ are in one-to-one correspondence with elements of $H^2(LS(\xi))$.
\end{corollary}
\begin{proof}
We need to analyse the homotopy commutative diagram of fibrations
\begin{gather}
\begin{aligned}
\xymatrix{
\Omega S(\xi) \ar[r]^{\Omega\rho}  \ar[d]  & \Omega M \ar[r]^{\Omega c} \ar[d]  & S^1 \ar[d]  \\
LS(\xi) \ar[r]^{L\rho}  \ar[d]  & LM \ar[r]^{Lc\ \ \ \ \ } \ar[d]  & S^1\times K(\mathbb{Z},2) \ar[d]  \\
S(\xi) \ar[r]^{\rho}  & M \ar[r]^{c\ \ \ \ \ }   & K(\mathbb{Z},2)  \\
}
\end{aligned}
\label{looplinefibrationcorodia}
\end{gather}
using the Serre spectral sequences. 
First, from the Serre spectral sequence (or Gysin sequence) of the fibration in the third row of Diagram \ref{looplinefibrationcorodia} there is a short exact sequence
\begin{equation}\label{surjeh2rhocoro2eq}
0\rightarrow H^2(K(\mathbb{Z},2))\stackrel{c^\ast}{\rightarrow} H^2(M)\stackrel{\rho^\ast}{\rightarrow} H^2(S(\xi))\rightarrow 0.
\end{equation}
On the other hand, since $M$ is simply connected, 
\[
H^2(M)\cong {\rm Hom}(H_2(M), \mathbb{Z})\cong {\rm Hom}(\pi_2(M),\mathbb{Z})
\]
is torsion free and $(\Omega c)_\ast: \pi_1(\Omega M)\rightarrow \pi_1(S^1)$ is surjective. Then the fibration in the top row of Diagram (\ref{looplinefibrationcorodia}) splits 
\[
\Omega M \simeq S^1\times \Omega S(\xi),
\]
which particularly implies that $(\Omega \rho)^\ast: H^2(\Omega M)\rightarrow H^2(\Omega S(\xi))$ is surjective. 
Now since $c_1(\xi)$ is a generator element of $H^2(M)$ by assumption, $S(\xi)$ is simply connected. We then can consider Serre spectral sequences of the fibrations in the first two columns of Diagram (\ref{looplinefibrationcorodia}). By the naturality of Serre spectral sequences and the fact that loop projection induces monomorphism on cohomology, we have the indued morphism of short exact sequences
\begin{gather}
\begin{aligned}
\xymatrix{
0\ar[r]   & H^2(M)  \ar[r]  \ar@{->>}[d]^{\rho^\ast}   &  H^2(LM)\ar[r] \ar[d]^{(L\rho)^\ast}  & H^2(\Omega M) \ar[r]  \ar@{->>}[d]^{(\Omega \rho)^\ast}  & 0\\
0 \ar[r]  & H^2(S(\xi)) \ar[r]                           & H^2(LS(\xi)) \ar[r]             & H^2(\Omega S(\xi))  \ar[r] & 0.
}
\end{aligned}
\label{maploopfibserressdia}
\end{gather}
Since we have showed that $\rho^\ast$ and $(\Omega\rho)^\ast$ are surjective, we see that the middle morphism $(L\rho)^\ast$ in the diagram is also surjective by the (sharp) five lemma. Hence the corollary follows.
\end{proof}


\section{Relations between strong and weak String$^c$ structures}\label{strongweakrelationsec}
\noindent In this section, we discuss the relations between strong and weak String$^c$ structures, and also the fusion Spin$^c$ structures on looped manifolds and their relations with String$^c$ structures.
\subsection{Strong String$^c$ vs. Weak String$^c$}
The relations between strong String$^c$- and weak String$^c$-structures are characterized by the following theorems:
\begin{theorem}\label{compa2stringcthm}
Let $V$ be an $n$-dimensional Spin$^c$-vector bundle over $M$ with a complex determinant line bundle $\xi$. If $V$ is strong String$^c$ of level $2k+1$, then $V$ is level $2k+1$ weak String$^c$. The converse is also true, if the image of the cohomology of the classifying map 
\[g^\ast: H^4(BSpin^c(n);\mathbb{Z})\rightarrow H^4(M;\mathbb{Z})\]
is a subgroup of the dual of the Hurewicz image $h: \pi_4(M)\rightarrow H_4(M;\mathbb{Z})$, and the rational Hurewicz morphism
\[h\otimes \mathbb{Q}: \pi_3(LM)\otimes \mathbb{Q} \rightarrow H_3(LM;\mathbb{Q})\]
is injective.
\end{theorem}
\begin{proof}
We use the free suspension $\nu$ to prove the theorem.
By Lemma \ref{suspenvalue}, for the universal case
\begin{equation*}
\nu\big(\frac{p_1-(2k+1)t_2^2}{2}\big)
= \nu(q_4 -kt_2^2)
= \mu_3-(2k+1)s_1t_2
=s_k(x_2).
\end{equation*}
By the naturality of $\nu$, we see that the obstructions to the weak and strong String$^c$-structures are connected via the equality
\begin{equation}\label{keyevalu}
\nu(\frac{p_1(V)-(2k+1)c^2}{2})=\mu_1(V)-(2k+1)sc.
\end{equation}
Hence the first claim of the theorem follows immediately. For the converse part of the theorem, we use the similar strategy used in the proof of Theorem $3.1$ in \cite{Mclau92} for the String case. The idea is to describe the free suspension $\nu$ geometrically at least for the elements in the Hurewicz image. Choose any $f\in \pi_4(M)$. $S^4$ can be covered by loops which meet only at one point (say the base point), and the parameter space for this set of loops is its equator $S^3$. By this view we obtain a class $g\in \pi_3(LM)$; indeed, this operation is equivalent to take the adjoint of $f$ to get $g\in \pi_3(\Omega M)$ and we notice that in general $\pi_\ast(LM)\cong\pi_\ast(M)\oplus \pi_\ast(\Omega M)$. In either way, this operation is the free supsension after taking the composition of Hurewicz map and the dual map, that is, we have the commutative diagram 
\begin{gather}
\begin{aligned}
\xymatrix{
\pi_4(M) \ar[r]^{h \ \  \ } \ar@{_{(}->}[d]^{i}  & H_4(M;\mathbb{Z})  \ar[r]^{{\rm dual}\ \ }  &H^4(M;\mathbb{Z}) \ar[d]^{\nu}  \\
\pi_3(LM) \ar[r]^{h   \ \ \ }                        & H_3(LM; \mathbb{Z}) \ar[r]^{{\rm dual}\ \ } & H^3(LM;\mathbb{Z}).
}
\end{aligned}
\label{hurewiczrelationdia}
\end{gather}
Now by assumption, the obstruction class $\frac{p_1(M)-(2k+1)c^2}{2}\in H^4(M)$ is from an element $f\in \pi_4(M)$. If $f$ is a torsion, then the dual of $h(f)$ will be $0$ (recall here dual is defined by the natural paring $H^\ast(M;\mathbb{Z})\times H_\ast(M;\mathbb{Z})\rightarrow \mathbb{Z}$). Otherwise $f$ is torsion free. Then by the above argument, we obtain an element $g=i(f)\in \pi_3(LM)$ such that $h(g)$ is non-zero by assumption. Take the dual of $h(g)$, we obtain the free suspension $\nu(\frac{p_1(M)-(2k+1)c^2}{2})$ which is non-zero. This is a contradiction, and then $\frac{p_1(M)-(2k+1)c^2}{2}=0$. The converse statement is proved.
\end{proof}

\begin{theorem}\label{compa2stringcthm2}
Let $(V,\xi)$ be as in Theorem \ref{compa2stringcthm}. Suppose $(V,\xi)$ is (strong) String$^c$ of level $2k+1$. Then the distinct strong String$^c$-structures lifting the original Spin$^c$-structure on $V$ transgress to the weak String$^c$-structures via the transgression $\nu$ 
\begin{gather}
\begin{aligned}
\xymatrix{
H^3(M) \ar[d]^{\nu}  \ar[r]^{\rho^\ast}   & H^3(S(\xi)) \ar[d]^{\nu}\\
H^2(LM) \ar[r]^{(L\rho)^\ast}                & H^2(LS(\xi)).
}
\end{aligned}
\label{countstrongweakstringcrelationdia}
\end{gather}
\end{theorem}
\begin{proof}
This follows immediately from the naturality of the involved constructions.
\end{proof}

\begin{corollary}\label{compa2stringcoro}
Let $(V,\xi)$ as in Theorem \ref{compa2stringcthm2}. Suppose $M$ is simply connected, and the Euler class $c_1(\xi)$ is a generator element of $H^2(M)$.
Then the distinct String$^c$-structures on $V$ transgress to the weak String$^c$-structures via the composition of the free suspension and the pullback
\[
(L\rho)^\ast\circ\nu=\nu\circ\rho^\ast: H^3(M) \rightarrow H^2(LS(\xi)).
\]
\end{corollary}
\begin{proof}
The corollary follows immediately from Theorem \ref{compa2stringcthm2}, Corollary \ref{rhosurjstringccor} and Corollary \ref{rhosurjweakstringccor}.
\end{proof}

\subsection{Fusive Spin$^c$ structures on looped manifolds}\label{fusionspincsubsec}
The String$^c$ structures can be also understood from the perspective of fusion structures, the study of which was initiated by Stolz and Teichner \cite{ST05}. In particular, they showed that an oriented manifold $N$ is Spin if and only if the loop space $LN$ is fusion orientable. Moreover, the equivalence classes of Spin structures on $N$ are in one-to-one correspondence with the fusion-preserving orientations of $LN$. If one drops the fusion conditions, the orientations on $LN$ can be viewed as weak Spin structures on $N$ in our terminology.

Similar results hold for the String case as well. For a Spin manifold $N$, a weak String structure one $LN$ can be defined as a lifting of the structure group of the looped frame bundle $LP_{Spin}(N)$ to the universal central extension $L\widehat{Spin}(n)$. It may be also called \textit{Spin structure on loop manifold} following Waldorf \cite{Wal16}, which was known as \textit{String structure on loop manifold} by the earlier work of Killingback \cite{K87} and McLaughlin \cite{Mclau92}.
In order to characterize String structures via Spin structures on loop manifolds, Waldorf \cite{Wal16} introduced additional fusion conditions and define the so-called \textit{fusion Spin structure} on $LN$, and proved that the universal central extension $L\widehat{Spin}(n)$ is a fusion extension in a canonical way. He then showed that $N$ is (strong) String if and only if $LN$ is fusion Spin. However, in this situation the fusion conditions are not enough to establish the bijection between the set of strong String structures and the set of fusion Spin structures, as remarked by Waldorf. Instead, he used thin homotopies of loops \cite{Wal12, Wal15} to investigate the correspondence. In contrast, Kottke-Melrose \cite{KM13} defined another modification of fusion Spin with some additional reparameterization equivariant conditions, which they called \textit{fusive loop Spin structures} over $LN$. 
They then showed that the equivalence classes of strong String structures on $N$ are in one-to-one correspondence with the equivalence classes of fusive loop Spin structures on $LN$. It should be noticed that all of these discussions are valid for general vector bundles with Spin structures.

Now let us consider the String$^c$ structures of negative levels on the Spin$^c$ manifold $(M, \xi)$. Let $k<0$ for the rest of this subsection. Recall that we have the $S^1$-invariant morphism of group extensions 
(\ref{Lgroupextmordia}). From that, the extension $L\widehat{Spin^c_k}(n)$ inherits a fusion structure from $L\widehat{Spin}(N)$. If the looped principal bundle $LP_{Spin}^k(M, \xi)$ of the vector bundle $TM\oplus \xi^{\oplus -2k-1}$ admits a fusion Spin structure in the sense of Waldorf, by Theorem \ref{weakstringcthm} and its proof, we may define the \textit{fusion Spin$^c$ structure of level $2k+1$} on $LM$ to be the restriction of the fusion Spin structure through the bundle embedding (cf. Diagram \ref{spinspincrelationweak})
\[
L\Theta_{2k+1}: LP_{Spin^c}(M)\hookrightarrow LP_{Spin}^k(M, \xi).
\]
Similarly we can also define \textit{fusive loop Spin$^c$ structures} on $LM$ of various levels following Kottke-Melrose. It is clear that if we drop the fusion conditions, the notion of the Spin$^c$ structures on $LM$ coincides with that of the weak String$^c$ structures on $M$.
Now recall by Theorem \ref{strongstringcgeothm}, strong String$^c$ structures can be also understood as strong String structures on $TM\oplus \xi^{\oplus -2k-1}$. Hence, the work of Waldorf \cite{Wal16} or Kottke-Melrose \cite{KM13} implies that 
$M$ is level $2k+1$ String$^c$ if and only if $LM$ is fusion (fusive loop) Spin$^c$ of level $2k+1$. Further by Kottke-Melrose \cite{KM13}, the equivalence classes of strong String$^c$ structures on $M$ are in one-to-one correspondence with the equivalence classes of fusive loop Spin$^c$ structures on $LM$.

Additionally, Kottke-Melrose \cite{KM13, KM15} defined the \textit{loop-fusion (\v{C}ech) cohomology}, $\check{H}^{\ast}_{lf}(LM;\mathbb{Z})$, and showed that the transgression map (i.e., the free suspension) $\nu$ factors through the isomorphic enhanced transgression $\nu_{lf}$
\begin{gather}
\begin{aligned}
\xymatrix{
\check{H}^{\ast}(M;\mathbb{Z})\ar[r]^{\nu_{lf} \ \ \ }_{\cong \ \ \ }  \ar[rd]_{\nu} & 
\check{H}^{\ast-1}_{lf}(LM;\mathbb{Z}) \ar[d]^{f}  \\ 
&  \check{H}^{\ast-1}(LM;\mathbb{Z}),
}
\end{aligned}
\label{enhancedtrandia}
\end{gather}
where $f$ is the forgetful morphism. Recall that the \v{C}ech cohomology is naturally isomorphic to the singular cohomology for CW complexes. The relations among strong String$^c$, weak String$^c$ and fusive loop Spin$^c$ then can be understood through this commutative diagram. Explicitly, the enhanced transgression of the obstruction to strong String$^c$ structure $\nu_{lf}(\frac{p_1(M)-(2k+1)c^2}{2})$ is the obstruction class to fusive loop Spin$^c$ structure, which reduces to $\mu_1(M)-(2k+1)sc$ the obstruction class to weak String$^c$ structure via the forgetful morphism $f$. Moreover, for the circle bundle $\rho: S(\xi)\rightarrow M$ of the determinant line bundle $\xi$, we have the commutative diagram
\begin{gather}
\begin{aligned}
 \xymatrix{
 \check{H}^{3}(M;\mathbb{Z}) \ar[r]^{\nu_{lf}}_{\cong} \ar@/^1.5pc/[rr]^{\nu} \ar[d]^{\rho^\ast}  &
  \check{H}^{2}_{lf}(LM;\mathbb{Z})  \ar[r]^{f}   \ar[d]^{(L\rho)^\ast}  &
  \check{H}^{2}(LM;\mathbb{Z})  \ar[d]^{(L\rho)^\ast}\\
   \check{H}^{3}(S(\xi);\mathbb{Z})\ar[r]^{\nu_{lf}}_{\cong}   \ar@/_1.5pc/[rr]_{\nu}  &
   \check{H}^{2}_{lf}(LS(\xi);\mathbb{Z})  \ar[r]^{f}    &
  \check{H}^{2}(LS(\xi);\mathbb{Z}),
 }
\end{aligned}
\label{enhancedtran2dia}
\end{gather}
where the outer rectangle is Diagram \ref{countstrongweakstringcrelationdia} in Theorem \ref{compa2stringcthm2}.
Then by Theorem \ref{strongstringcgeothm} and Theorem \ref{weakstringcthm}, we see that the equivalence classes of strong String$^c$ structures on $M$ transgress to the equivalence classes of fusive loop Spin$^c$ structures on $LM$, and then to the weak String$^c$-structures.

Let us summarise the above discussions in the following theorem. For details of the precise definitions of various fusion structures, loop-fusion (\v{C}ech) cohomology and others, please refer to Waldorf \cite{Wal16}, and Kottke-Melrose \cite{KM13, KM15}.
\begin{theorem}
Let $M$ be a connected compact Spin$^c$ manifold. Let $k<0$. Then
\begin{itemize}
\item[(1).] $M$ is level $2k+1$ String$^c$ if and only if $LM$ is fusion (or fusive loop) Spin$^c$ of level $2k+1$;
\item[(2).] the equivalence classes of strong String$^c$ structures on $M$ are in one-to-one correspondence with the equivalence classes of fusive loop Spin$^c$ structures on $LM$;
\item[(3).] the equivalence classes of strong String$^c$ structures on $M$ transgress to the equivalence classes of fusive loop Spin$^c$ structures on $LM$ through the enhanced transgression, and then to the weak String$^c$-structures after composing the forgetful map.
\end{itemize}
\end{theorem}


\section{Modular invariants and group actions on String$^c$ manifolds}\label{genwittengenussec}
\noindent In this section, for even dimensional level $(2k+1)$ String$^c$ manifolds with $2k+1>0$, we construct Witten type genera, which are modular invariants taking values in $\Z [\frac{1}{2}]$ and prove Liu's type vanishing theorem for them. 
They extend the generalized Witten genera for level 1 and level 3 String$^c$ manifolds constructed in \cite{CHZ10, CHZ11}. We also give some applications of these vanishing results to Lie group actions on manifolds.

\subsection{Generalized Witten genera and vanishing theorems}\label{genwittensubsec}
Let $M$ be a Spin$^c$ manifold, which is level $2k+1$ String$^c$ with $2k+1>0$. Let 
\[
\vec a=(a_1, a_2, \cdots, a_{r})\in \Z^r,  ~~~~ \ \ \vec b=(b_1, b_2, \cdots, b_s) \in \Z^s
\]
be two vectors of integers such that $\sum_{j=1}^r a_j+\sum_{j=1}^s b_j $ is even. If $M$ is $4m$ dimensional, we require that
\be \label{cond4m} 3||\vec a||^2+||\vec b||^2=2k-2; \ee
and if $M$ is $4m+2$ dimensional, we require that
\be \label{cond4m+2} 3||\vec a||^2+||\vec b||^2=2k. \ee

Let $\xi$ be the determinat line bundle of the Spin$^c$ structure. 
Let $h^\xi$ be a Hermitian metric on $\xi$ and $\nabla^\xi$ be a Hermitian
connection. Let $h^{\xi_{\R}}$ and $\nabla^{\xi_{\R}}$ be the
induced Euclidean metric and connection on $\xi_{ \R}$. Construct
\begin{equation}
\begin{split}
&\Theta_{\vec a, \vec b}(T_{\mathbb{C}}M,\xi_{\mathbb{R}}\otimes \mathbb{C})\\
:=& \left( \overset{%
\infty }{\underset{n=1}{\bigotimes }}S_{q^{2n}}(\widetilde{T_{\mathbb{C}}M}%
)\right) \\ 
&\otimes \otimes_{j=1}^r \left( \overset{\infty }{\underset{n=1}{\bigotimes }}\Lambda
_{q^{2n}}(\widetilde{\xi^{\otimes a_j}_{\mathbb{R}}\otimes \mathbb{C}})\otimes \overset{\infty }{\underset{n=1}{\bigotimes }}\Lambda
_{-q^{2n-1}}(\widetilde{\xi^{\otimes a_j}_{\mathbb{R}}\otimes \mathbb{C}})\otimes \overset{\infty }{\underset{n=1}{\bigotimes }}\Lambda _{q^{2n-1}}(%
\widetilde{\xi^{\otimes a_j}_{\mathbb{R}}\otimes \mathbb{C}}) \right) \\
&\otimes \otimes_{j=1}^s \left( \overset{\infty }{\underset{n=1}{\bigotimes }}%
\Lambda _{-q^{2n}}(\widetilde{\xi^{\otimes b_j}_{\mathbb{R}}\otimes \mathbb{C}})\right),
\end{split}
\end{equation}
where $\widetilde{E}=E-\mathbb{C}^{{\rm dim}(E)}$ for any complex bundle $E$.
Then $\nabla^{TM}$ and $\nabla^{\xi}$ induce connections
$\nabla^{\Theta_{\vec a, \vec b}(T_{\mathbb{C}}M,\xi_{\mathbb{R}}\otimes \mathbb{C})}$ on $\Theta_{\vec a, \vec b}(T_{\mathbb{C}}M,\xi_{\mathbb{R}}\otimes \mathbb{C})$. Let $c=c_1(\xi, \nabla^{\xi})$ be the
first Chern form of $(\xi, \nabla^{\xi})$.

If dim$M=4m$, define the {\em type $(2k+1; \vec a, \vec b)$ Witten form}
\be
\begin{split}
&\mathcal{W}^c_{2k+1; \vec a, \vec b}(M)\\
:=&\widehat{A}(TM, \nabla^{TM})e^{\frac{c}{2}}\prod_{j=1}^r\cosh\left(\frac{a_jc}{2}\right)
\prod_{j=1}^s\sinh\left(\frac{b_jc}{2}\right)\\
&\cdot \mathrm{ch}\left( \Theta_{\vec a, \vec b} (T_{\mathbb{C}}M,\xi_{\mathbb{R}}\otimes \mathbb{C}) \otimes \overset{\infty }{\underset{n=1}{\bigotimes }}\Lambda
_{q^{2n}}(\widetilde{\xi_{\mathbb{R}}\otimes \mathbb{C}})\otimes \overset{\infty }{\underset{n=1}{\bigotimes }}\Lambda
_{-q^{2n-1}}(\widetilde{\xi_{\mathbb{R}}\otimes \mathbb{C}})\otimes \overset{\infty }{\underset{n=1}{\bigotimes }}\Lambda _{q^{2n-1}}(%
\widetilde{\xi_{\mathbb{R}}\otimes \mathbb{C}}) \right).\\
\end{split} 
\ee
If dim$M=4m+2$, define the {\em type $(2k+1; \vec a, \vec b)$ Witten form}
\be
\begin{split}
&\mathcal{W}^c_{2k+1; \vec a, \vec b}(M)\\
:=&\widehat{A}(TM, \nabla^{TM})e^{\frac{c}{2}}\prod_{j=1}^r\cosh\left(\frac{a_jc}{2}\right)
\prod_{j=1}^s\sinh\left(\frac{b_jc}{2}\right)\mathrm{ch}\left( \Theta_{\vec a, \vec b} (T_{\mathbb{C}}M,\xi_{\mathbb{R}}\otimes \mathbb{C}) \otimes \overset{\infty }{\underset{n=1}{\bigotimes }}\Lambda
_{q^{-2n}}(\widetilde{\xi_{\mathbb{R}}\otimes \mathbb{C}})\right).
\end{split} 
\ee

We can express these generalized Witten forms by using the Chern-root algorithm.
Let $\{\pm 2\pi \sqrt{-1}z_{j}\}$ be the formal Chern roots for $(T_{\mathbb{%
C}}M,\nabla ^{T_{\mathbb{C}}M})$ and set $u=-\frac{\sqrt{-1}}{2\pi}c$. In terms of the theta-functions (the details about which are discussed in Appendix \hyperref[AppendixD]{D}), we get
through direct computations that (c.f. \cite{Liu95, Liu96, CHZ10, CHZ11})

\begin{equation}\label{Eqn: Witten genus Chern roots 4m}
\begin{split}
&\mathcal{W}^c_{2k+1, \vec a, \vec b}(M^{4m})\\
=& \left(
\prod_{j=1}^{2m}z_{j}\frac{\theta ^{\prime }(0,\tau )}{\theta (z_{j},\tau )}%
\right) \frac{\theta _{1}(u,\tau )\theta _{2}(u,\tau )\theta _{3}(u,\tau )}{%
\theta _{1}(0,\tau )\theta _{2}(0,\tau )\theta _{3}(0,\tau )}\prod_{j=1}^r\frac{\theta _{1}(a_ju,\tau )\theta _{2}(a_ju,\tau )\theta _{3}(a_ju,\tau )}{%
\theta _{1}(0,\tau )\theta _{2}(0,\tau )\theta _{3}(0,\tau )}\\
&\cdot \prod_{j=1}^s\frac{\sqrt{-1}\theta (b_ju,\tau )}{\theta _{1}(0,\tau )\theta
_{2}(0,\tau )\theta _{3}(0,\tau )}\\
\end{split}  
\end{equation}
and
\begin{equation} \label{Eqn: Witten genus Chern roots 4m+2}
\begin{split}
&\mathcal{W}^c_{2k+1, \vec a, \vec b}(M^{4m+2})\\
=& \left(
\prod_{j=1}^{2m}z_{j}\frac{\theta ^{\prime }(0,\tau )}{\theta (z_{j},\tau )}%
\right) \prod_{j=1}^r\frac{\theta _{1}(a_ju,\tau )\theta _{2}(a_ju,\tau )\theta _{3}(a_ju,\tau )}{%
\theta _{1}(0,\tau )\theta _{2}(0,\tau )\theta _{3}(0,\tau )}\\
&\cdot \frac{\sqrt{-1}\theta (u,\tau )}{\theta _{1}(0,\tau )\theta
_{2}(0,\tau )\theta _{3}(0,\tau )}\prod_{j=1}^s \frac{\sqrt{-1}\theta (b_ju,\tau )}{\theta _{1}(0,\tau )\theta
_{2}(0,\tau )\theta _{3}(0,\tau )}.\\
\end{split}  
\end{equation}

Define the {\em type $(2k+1; \vec a, \vec b)$ Witten genus} by
\begin{equation}
W^c_{2k+1, \vec a, \vec b}(M^{4m}):=\int_{M^{4m}} \mathcal{W}^c_{2k+1, \vec a, \vec b}(M^{4m}),
\end{equation}
and
\begin{equation}
W^c_{2k+1, \vec a, \vec b}(M^{4m+2}):=\int_{M^{4m+2}} \mathcal{W}^c_{2k+1, \vec a, \vec b}(M^{4m+2}).
\end{equation}

Note that 
\be \prod_{j=1}^r\cosh\left(\frac{a_jc}{2}\right)\prod_{j=1}^s\sinh\left(\frac{b_jc}{2}\right)=\frac{1}{2^{r+s}}e^{-\frac{\sum_{j=1}^r a_j+\sum_{j=1}^s b_j }{2}c}\prod_{j=1}^r\left(e^{a_jc}+1\right)\prod_{j=1}^s\left(e^{b_jc}-1\right).\ee
However since $\sum_{j=1}^r a_j+\sum_{j=1}^s b_j $ is even, one has that $e^{-\frac{\sum_{j=1}^r a_j+\sum_{j=1}^s b_j }{2}c}\prod_{j=1}^r\left(e^{a_jc}+1\right)\prod_{j=1}^s\left(e^{b_jc}-1\right)$ is the Chern character of some vector bundle. Hence by the Atiyah-Singer index theorem, $2^{r+s}W^c_{2k+1, \vec a, \vec b}(M^{4m})$ and  $2^{r+s}W^c_{2k+1, \vec a, \vec b}(M^{4m+2})$ are analytic, i.e., they are indices of $q$-series of twisted Spin$^c$ Dirac operators. We therefore see that 
$W^c_{2k+1; \vec a, \vec b}(M^{4m})\in \Z[\frac{1}{2}]$ and $W^c_{2k+1; \vec a, \vec b}(M^{4m+2})\in \Z[\frac{1}{2}]$. 

By the same method in \cite{Liu95cmp}, using the conditions (\ref{cond4m}) or (\ref{cond4m+2}) when performing the transformation laws of theta functions, we have
\begin{theorem}
\label{Thm: Modularity of Witten type genus}If $\dim M=4m$, then $W^c_{2k+1; \vec a, \vec b}(M^{4m})\in \Z[\frac{1}{2}]$ is a modular form of weight $2m$ over $SL(2,\mathbb{Z})$; if $\dim M=4m+2$, then $W^c_{2k+1; \vec a, \vec b}(M^{4m+2})\in \Z[\frac{1}{2}]$ is a modular form of weight $2m$ over $SL(2,\mathbb{Z})$.   \hfill $\Box$
\end{theorem}

For the generalized Witten genus $W^c_{2k+1; \vec a,  \vec b}(M)$, we have the following Liu's type vanishing theorem.

\begin{theorem}
\label{Thm:action} Let $M$ be a connected compact level $2k+1$ String$^c$ manifold with $2k+1>0$. If $M$ admits an effective action of a simply connected compact Lie group that can be lifted to the Spin$^c$ structure and the action is positive, then $W^c_{2k+1; \vec a,  \vec b}(M)=0$. 
\end{theorem}
\begin{proof} Let $G$ be the simply connected compact Lie group. It has been shown in \cite{MP86} that $G$ contains $SU(2)$ or $SO(3)$ as subgroup. Since there exists the standard $2$-sheet covering $p: SU(2)\to SO(3)$, in either case we see that there exists a $SU(2)$-action on $M$ factoring through $G$. Then choose any subgroup $S^1\hookrightarrow SU(2)$. Since $G$ acts effectively on $M$ and can be lifted to the Spin$^c$ structure, we have that the induced $S^1$-action is non-trivial and can be lifted to the bundle $TM- \xi^{\oplus (2k+1)}$. 
In particular, through the induced composition map of classifying spaces
\[
BS^1\rightarrow BSU(2)\rightarrow BG,
\]
there exists the canonical generator $q\in H^4(BG)$ restricted to the generator $u^2\in H^4(BS^1)$.

We now can apply the similar argument of Dessai \cite{Des94} to the bundle $TM- \xi^{\oplus (2k+1)}$. 
If the $S^1$-action has no fixed points, the generalized Witten genus vanishes by the Atiyah-Bott-Segal-Singer-Lefschetz fixed point formula (\cite{AB88, AS}). Otherwise suppose that there are some fixed points. Let $EG$ be the universal $G$-principal bundle over the classifying space $BG$ of any topological group $G$. By applying the dual Blakers-Massey theorem (Theorem \ref{dualBM} in Appendix \hyperref[AppendixC]{C}) to the Borel fibre bundle
\[
M\stackrel{i}{\rightarrow} M\times_{G}EG\stackrel{\pi}{\rightarrow} BG
\]
(with the fact that $BG$ is $3$-connected), we see that there exists a commutative diagram
\begin{gather*}
\begin{aligned}
\xymatrix{
0\ar[r]  &   H^4(BG) \ar[d]  \ar[r]^{\pi^\ast \ \ \ \ \ \ \ }   & H^4(M\times_{G}EG) \ar[r]^{   \ \ \    \ \ \  \ \ \  \ \ \ i^\ast}  \ar[d] &H^4(M)    \ar@{=}[d]  \\
  & H^4(BS^1) \ar[r]^{\pi^\ast\ \ \ }  & H^4(M\times_{S^1}ES^1) \ar[r]^{   \ \ \  \ \ \    \ \ \  i^\ast} & H^4(M), \\
  }
\end{aligned}
\end{gather*}
such that the first row is exact, and maps to the second row by restricting the action to $S^1$.
On the other hand, since the level $2k+1$ String$^c$ condition tells us that $p_1(TM- \xi^{\oplus (2k+1)})=0$,
we have 
\[
p_1(TM)_{G}-(2k+1)c_1(\xi)^2_{G}=p_1(TM- \xi^{\oplus (2k+1)})_{G}=n\cdot \pi^\ast q
\] 
for some $n>0$ by positive assumption (\ref{positivedefintro}).
Hence by the above commutative diagram we see that the restriction of the equivariant Pontryagin class 
\[
p_1(TM)_{S^1}-(2k+1)c_1(\xi)^2_{S^1}=p_1(TM- \xi^{\oplus (2k+1)})_{S^1}=n\cdot \pi^\ast u^2, \ \ n>0.
\]
The theorem then follows by the proof of Liu's vanishing theorem \cite{Liu95} for nonzero anomaly about Witten genus. 
\end{proof}

\subsection{Some applications of the vanishing theorem}\label{appsubsec}

Suppose $(M, J)$ is a compact stable almost complex manifold. Then $M$ has a canonical Spin$^c$ structure determined by $J$. If $G$ acts smoothly on $M$ and preserves the stable almost complex structure $J$, then the action of $G$ can be lifted to the Spin$^c$ structure and the bundle $\xi$, the determinant line bundle. Applying the above vanishing theorem, we immediately obtain
\begin{theorem}[\protect Theorem \ref{Mainstableacs}]
Let $(M, J)$ be a compact stable almost complex manifold, which is level $2k+1$ String$^c$, i.e., $p_1(TM)=(2k+1)c_1^2$ and suppose $2k+1>0$. If $M$ admits a positive effective action of $G$ preserving $J$, then $W^{c_1}_{2k+1; \vec a,  \vec b}(M)=0$.
\end{theorem}

Recall that our generalized Witten genera are indexed by the pair of vectors $(\vec a,  \vec b)$. It turns out that this flexibility allows us to deduce results concerning the group actions on manifolds with particular arithmetic conditions. In particular, we prove a slightly stronger version of Corollary \ref{genexamplethmintro}.
\begin{theorem}\label{genexamplethmsec}
Let $(M, \xi, c=c_1(\xi))$ be a compact String$^c$ manifold of level $2k+1$ with the determinant line bundle $\xi$. 
If $M$ satisfies one of the following 
\begin{itemize}
\item[$(A)$] $M=M^{4m}$, and $c^{2m}\neq 0$ rationally,
  \begin{itemize}
  \item[$(A.1)$] $k-m\equiv 0~{\rm mod}~3$, and $k-m\geq 9$,
  \item[$(A.2)$] $k-m\equiv 1~{\rm mod}~3$, and $k-m\geq 1$,
  \item[$(A.3)$] $k-m\equiv -1~{\rm mod}~3$, and $k-m\geq 5$;
  \end{itemize}
\item[$(B)$] $M=M^{4m+2}$, and $c^{2m+1}\neq 0$ rationally,
  \begin{itemize}
  \item[$(B.1)$] $k-m\equiv 0~{\rm mod}~3$, and $k-m\geq 0$,
  \item[$(B.2)$] $k-m\equiv 1~{\rm mod}~3$, and $k-m\geq 4$,
  \item[$(B.3)$] $k-m\equiv -1~{\rm mod}~3$, and $k-m\geq 8$,
  \end{itemize}
\end{itemize}
then $M$ does not admit a positive effective action of a simply connected compact Lie group that can be lifted to the underlying Spin$^c$ structure.
\end{theorem}
\begin{proof}
Suppose $M=M^{4m}$, and consider the quadratic indefinite equation
\begin{equation}\label{general4mconeq}
3||\vec a||^2+b_1^2+b_2^2+\cdots +b_{2m}^2=2k-2,
\end{equation}
where we let $s=2m$. Let $\vec b=(b_1, b_2, \ldots, b_{2m})=(1,1,\ldots ,1,3,3)$ for Case $(A.1)$, $\vec b=(1,1,\ldots,1)$ for Case $(A.2)$, and $\vec b=(1,1,\ldots,1,3)$ for Case $(A.3)$ respectively.
In particular, we see that $||\vec b||^2=2m+16$, $2m$, or $2m+8$ in each case, the collection of which is a complete residue system modulo $3$. Then it is easy to check that
\[
2k-2-||\vec b||^2\equiv 0~{\rm mod}~3
\]
in each of the three cases, where the left hand side is also non-negative. Hence the equation (\ref{general4mconeq}) always has an integer solution for some $\vec a$ by Lagrange's four-square theorem.

By (\ref{Eqn: Witten genus Chern roots 4m}), we have 
\begin{equation*}
\begin{split}
&W^{c}_{2k+1; \vec a, \vec b}(M^{4m})\\
=& \int_{M^{4m}}\left(
\prod_{j=1}^{2m}z_{j}\frac{\theta ^{\prime }(0,\tau )}{\theta (z_{j},\tau )}%
\right) \frac{\theta _{1}(u,\tau )\theta _{2}(u,\tau )\theta _{3}(u,\tau )}{%
\theta _{1}(0,\tau )\theta _{2}(0,\tau )\theta _{3}(0,\tau )}\prod_{j=1}^r\frac{\theta _{1}(a_ju,\tau )\theta _{2}(a_ju,\tau )\theta _{3}(a_ju,\tau )}{%
\theta _{1}(0,\tau )\theta _{2}(0,\tau )\theta _{3}(0,\tau )}\\
&\cdot \left(\frac{\sqrt{-1}\theta (u,\tau )}{\theta _{1}(0,\tau )\theta
_{2}(0,\tau )\theta _{3}(0,\tau )}\right)^p\left(\frac{\sqrt{-1}\theta (3u,\tau )}{\theta _{1}(0,\tau )\theta
_{2}(0,\tau )\theta _{3}(0,\tau )}\right)^q,\\
\end{split}  
\end{equation*}
where $u=-\frac{\sqrt{-1}}{2\pi}c$, and $p$, $q$ are non-negative integers depending on the different cases, but alway satisfy $p+q=2m$.
Note that $\theta(\nu, \tau)$ is an odd function of $\nu$ starting from $2\pi \sqrt{-1}\nu$ in the expansion while $\theta_i(\nu, \tau)$ are all even functions of $\nu$ for $i=1,2,3$, we have  
\[ W^{c}_{2k+1; \vec a, \vec b}(M^{4m})=\int_{M^{4m}} d\cdot c^{2m}\neq 0,\]
where $d$ is some non-zero constant. Hence by Theorem \ref{Thm:action}, we see that $M^{4m}$ admits no effective positive action of a compact non-abelian Lie group that can be lifted to $\xi$. For the cases when $M=M^{4m+2}$, the proof is similar and omitted. We then obtain the theorem.
\end{proof}

Our vanishing theorem can be applied to study group actions on homotopy complex projective spaces.
Let $X$ be a closed smooth manifold homotopic to $\CC P^{2n}$. Let $x\in H^2(X; \Z)$ be a generator. Using his twisted Spin$^c$ rigidity theorem, Dessai \cite{Des99} proved the following

\begin{theorem}[\protect Dessai]\label{homotopyprojsec}
Let $X$ be a closed smooth manifold homotopic to $\CC P^{2n}$. If $p_1(X)>(2n+1)x^2$, then $X$ does not support a nontrivial smooth $S^3$ action.
\end{theorem}

We give a proof of this theorem by using the vanishing of the generalized Witten genera. 

\begin{proof} By Masudai-Tsai \cite{MT85}, one knows that the first Pontryagin class of $X$ takes the form $p_1(X)=(2n+1+24\alpha(X))x^2$ for certain integer $\alpha(X)$. Therefore $X$ is String$^c$ of level $2n+1+24\alpha(X)$ with the underlying Spin$^c$ structure determined by $x$. By the assumption $p_1(X)>(2n+1)x^2$, we have $\alpha(X)>0$. 
And therefore, similar to the proof of Theorem \ref{genexamplethmsec}, the indefinite equation 
\[
3||\vec a||^2+b_1^2+b_2^2+\cdots +b_{2n}^2=2n-2+24\alpha(X)
\]
must have a solution $(\vec a, \vec b)$ such that all the $b_i's$ are nonzero.
It implies that $W^{c}_{2n+1+24\alpha(X); \vec a, \vec b}(X)$ is well defined, and again by similar argument as in the proof of Theorem \ref{genexamplethmsec}, it can be showed that this general Witten genus does not vanish.

Now assume that there is a nontrivial $S^3$ action on $X$. First by \cite{HY76}, this action can be be lifted to the determinant line bundle determined by $x$. Also by Lemma $3.8$ of \cite{Des99}, since $X$ has odd Euler characteristic, the induced action of the subgroup $Pin(2)$ of $S^3$ on $X$ has a fixed point. Hence from Remark \ref{pin2}, we see that the $S^3$ action is positive. 
In addition, from the fact that $S^3$ is covered by the conjugate classes of its maximal torus $T\cong S^1$, it follows that there exists a subgroup $S^1$ of $S^3$ which acts nontrivially on $X$. 
Consequently by the proof of Theorem \ref{Thm:action}, we see that $W^{c}_{2n+1+24\alpha(X); \vec a, \vec b}(X)$ must vanish, which is a contradiction. 
\end{proof}


\appendix
\section{Basics on homotopy fibre sequences}
\label{AppendixA}

\noindent For any pointed map $f: X\rightarrow Y$, there is a canonical way to turn it into a fibration with a homotopy fibre $F$
\[
F\stackrel{i}{\rightarrow} X\stackrel{f}{\rightarrow} Y.
\]
Continue the process for the leftmost maps, we then obtain the so-called \textit{Puppe sequence} of $f$ (e.g. See Chapter $2$ of \cite{Switzer75}) 
\[
\cdots\stackrel{\Omega j}{\rightarrow} \Omega F\stackrel{\Omega i}{\rightarrow} \Omega X\stackrel{\Omega f}{\rightarrow} \Omega Y\stackrel{j}{\rightarrow}  F\stackrel{i}{\rightarrow} X\stackrel{f}{\rightarrow} Y, 
\]
of which any three consecutive terms give a homotopy fibration.
The following lemma is used frequently in this paper without further reference:
\begin{lemma}[Lemma $2.1$ of \cite{Cohen79}]
A homotopy commutative diagram
\[
\xymatrix{
A \ar[r] \ar[d]  & B \ar[d]\\
C \ar[r]          & D
}
\]
can be embedded in a homotopy commutative diagram
\[
\xymatrix{
Q \ar[r] \ar[d]   & J\ar[r] \ar[d]  & K \ar[d]\\ 
F\ar[r] \ar[d]   &A \ar[r] \ar[d]  & B \ar[d]\\
G\ar[r]            &C \ar[r]          & D,
}
\]
in which the rows and columns are fibration sequences up to homotopy.
\end{lemma}

\section{Cohomology suspension and transgression}\label{AppendixB}
\noindent In cohomology theory there are two classical kinds of suspensions (e.g., see Section $1.3$ of \cite{Harper02}): \textit{Mayer-Vietoris suspension (MV-suspension)}
\begin{equation}\label{mvsuspen}
\Delta^\ast: \bar{H}^n(X)\rightarrow H^{n+1}(\Sigma X),
\end{equation}
and \textit{cohomology suspension}
\begin{equation}\label{cohosuspen}
\sigma^\ast: H^{n+1}(X)\rightarrow H^{n}(\Omega X).
\end{equation}
The MV-suspension $\Delta^\ast$ is also known as part of the axioms of general reduced cohomology theories and is always an isomorphism. The cohomology suspension then does not hold in general, and can be defined as 
\begin{equation}\label{defsigma1}
\sigma^\ast: H^{n+1}(X)\stackrel{p}{\rightarrow} H^{n+1}(PX, \Omega X) \xleftarrow[\cong]{\delta} H^n(\Omega X),
\end{equation}
where $p: (PX, \Omega X) \rightarrow (X,\ast)$ is the canonical path fibration, $\delta$ is the connecting homomorphism in the long exact sequence of the cohomology of the pair $(PX, \Omega X)$.

There are other two useful alternative descriptions. Firstly we may identify cohomology groups with groups of homotopy classes of maps into Eilenberg-Maclane spaces via the Brown representability theorem
\begin{equation}\label{brownrep}
\bar{H}^n(X;\mathbb{Z})\cong [X, K(\mathbb{Z}, n)].
\end{equation}
Then the $MV$-suspension is just to take the adjoint map and the cohomology suspension is to take the loop functor 
\begin{equation}\label{defsigma2}
\Omega: [X, K(\mathbb{Z}, n+1)]\rightarrow [\Omega X, K(\mathbb{Z}, n)].
\end{equation}

We may also define the cohomology suspension via the evaluation map 
\begin{equation}\label{evloop}
{\rm ev}: S^1\times \Omega X\rightarrow X
\end{equation}
defined by ${\rm ev}((t, \omega))=\omega(1)$. In this case, $\sigma^\ast$ is a slant-product by the fundamental class $[S^1]$ of $S^1$
\begin{equation}\label{defsigma3}
{\rm ev}^\ast(x)=s_1\otimes \sigma^\ast(x).
\end{equation}

Both the MV-suspension and the cohomology suspension are natural and have a useful connection, that is, 
\begin{equation}\label{conn2suspen}
\Delta\circ \sigma^\ast= \bar{{\rm ev}}^\ast: H^{n+1}(X)\rightarrow H^{n+1}(\Sigma \Omega X),
\end{equation}
where $\bar{{\rm ev}}: \Sigma \Omega X\rightarrow X$ is the (reduced) evaluation map. In particular, $\sigma^\ast$ is trivial on decomposable elements since the ring structure of the cohomology of a suspension is trivial.

We should be careful to use cohomology suspension when $n=0$ or $X$ is not simply connected. In these cases, we may define the \textit{$k$-th component cohomology suspension} of $\sigma^\ast$ by 
\begin{equation}\label{defkcohosuspen}
\sigma_k^\ast: H^{n+1}(X)\stackrel{\sigma^\ast}{\rightarrow} H^n(\Omega X) \stackrel{(i_k)^\ast}{\rightarrow} H^n(\Omega_k X),
\end{equation} 
where $i_k: \Omega_k X\hookrightarrow \Omega X$ is the inclusion of the $k$-th component of $\Omega X$ for $k\in \pi_0(\Omega X)$.
The other two equivalent definitions of $\sigma^\ast_k$ can be easily obtained from (\ref{defsigma2}) and (\ref{defsigma3}).

\begin{example}\label{exsuspens1}
Let us compute 
\[
\sigma^\ast: H^1(S^1)\rightarrow H^0(\Omega S^1),
\]
which is equivalent to 
\[
\Omega: [S^1 , S^1]\rightarrow \langle \Omega S^1, \Omega S^1\rangle,
\]
where $\langle -, - \rangle$ denotes the set of homotopy classes of free maps. We notice that there are group isomorphisms
\[\langle\Omega S^1, \Omega S^1\rangle\cong {\rm Func} (\mathbb{Z},\mathbb{Z})\cong \prod_{k\in \mathbb{Z}} {\rm Func} (k,\mathbb{Z}),\]
where ${\rm Func}(-,-)$ denotes the set of functions and the group structure of $\prod_{k\in \mathbb{Z}} {\rm Func} (k,\mathbb{Z})$ is defined pointwise and inherited from the targets $\mathbb{Z}$. Further combining with the Brown representability theorem, $H^0(\Omega_k(S^1))$ corresponds exactly to ${\rm Func} (k,\mathbb{Z})$.
Since $\Omega({\rm id})={\rm id}$ corresponds to $\prod_{k\in \mathbb{Z}}(\lambda_k: k\mapsto k)$, we see that 
\[
\sigma_k^\ast(s_1)=k.
\]
\end{example}

The cohomology suspension has a ``partial'' inverse, known as cohomology transgression (e.g. see Section $6.2$ of \cite{Mccl01} or Section \MyRoman{13}$.7$ of \cite{Whitehead78}). For simplicity let us introduce it directly by the Serre spectral sequence $(E_r^{\ast,\ast}, d_r)$ of any given orientable fibration $F \stackrel{i}{\rightarrow} E \stackrel{p}{\rightarrow} B$.
\begin{definition}\label{transdef}
The \textit{cohomology transgression} is the differential homomorphism 
\begin{equation}\label{transdefmor}
d_n: E_{n}^{0,n-1}\rightarrow E_{n}^{n,0}
\end{equation}
for each $n\geq 2$.
\end{definition}
The cohomology transgression fits into following commutative diagram 
\begin{gather}
\begin{aligned}
\xymatrix{
H^{n-1}(E) \ar[r]^{i^\ast}  \ar@{->>}[d]    & H^{n-1}(F) \ar[r]^{\delta}    & H^{n}(E, F)\ar[r]^{j^\ast}   & H^{n}(E) \\
E_{\infty}^{0,n-1}\cong E_{n+1}^{0,n-1} \ar@{^{(}->}[r]    &  E_{n}^{0,n-1}\ar[r]^{d_n}  \ar@{^{(}->}[u]  & E_{n}^{n,0}   \ar@{^{(}->}[u] 
\ar@{->>}[r]  &  E_{n+1}^{n,0}\cong  E_{\infty}^{n,0}  \ar@{^{(}->}[u]\\
&& H^{n}(B,\ast) \ar@{->>}[u]  \ar@/_1.5pc/[uu]_>(.8){ p^\ast} \ar[r]^{j^\ast}  & H^n(B),  \ar@{->>}[u]   \ar@/_2pc/[uu]_>(.8){ p^\ast}
}
\end{aligned}
\label{transdiag}
\end{gather}
where the first line is part of the long exact sequence of the cohomology of the pair $(E, F)$, and the second row is exact by the definition of $d_n$. Then it is easy to show that $d_n$ can be described as a homomorphism 
\begin{equation}\label{transdef0}
\tau: H^{n-1}(F)\supseteq \delta^{-1}({\rm Im}~p^\ast) \rightarrow H^{n}(B)/ j^\ast({\rm Ker}~p^\ast).
\end{equation}
To consider the connection to cohomology suspension, we specify the above argument to the loop fibration $\Omega X\rightarrow PX\rightarrow X$. In this case both $d_n$ and $\delta$ are isomorphisms and the composition $\delta^{-1}\circ p^\ast: H^{n}(X)\rightarrow H^{n-1}(\Omega X)$ is exactly the cohomology suspension $\sigma^\ast$ by definition. Hence we see that $\tau$ is a partial inverse of $\sigma^\ast$.

\section{Blakers-Massey type theorems}\label{AppendixC}
\begin{definition}\label{condef}
Let $f: X\rightarrow Y$ be a pointed map between pointed spaces $X$ and $Y$. Then 
$f$ is \textit{$n$-connected} if it induces isomorphisms on $k$-dimensional homotopy groups for $k<n$ and an epimorphism for $k=n$. The space $X$ is $m$-connected if $\pi_i(X)=0$ for any $i\leq m$. We use the convention that any space is $(-1)$-connected.
\end{definition}
It is then easy to check that $f$ is $n$-connected is equivalent to any of the following:
\begin{itemize} 
\item[(1)] the homotopy fibre of $f$ is $(n-1)$-connected;
\item[(1)] the homotopy cofibre of $f$ is $n$-connected;
\item[(2)] $f_\ast: H_i(X;\mathbb{Z})\rightarrow H_i(Y;\mathbb{Z})$ is an isomorphism for each $k<n$ and an epimorphism for $k=n$; 
\item[(3)] $f^\ast: H^i(Y;\mathbb{Z})\rightarrow H^i(X;\mathbb{Z})$ is an isomorphism for each $k<n$ and a monomorphism for $k=n$.
\end{itemize}

\begin{theorem}[An elegant form of Blakers-Massey Theorem; e.g., see Theorem $4.2.1$ \cite{MV15}]\label{BM}
Let 
\[
\xymatrix{
B \ar[r]^{f}  \ar[d]^{g}   & A  \ar[d]^{h}   \\
C \ar[r]^{k}                  & X    \pushoutcorner
}
\]
be a homotopy pushout diagram. 
Let 
\[
\xymatrix{
Y   \pullbackcorner \ar[r]  \ar[d]   & A  \ar[d]^{h}   \\
C \ar[r]^{k}                  & X  
}
\]
be the homotopy pullback diagram defining $Y$. Suppose $f$ is $m$-connected and $g$ is $n$-connected. Then the induced map $B\rightarrow Y$ is $(m+n-1)$-connected.
\end{theorem}

\begin{theorem}[Dual Blakers-Massey Theorem of fibrations; a folklore theorem for homotopy theorists]\label{dualBM}
Let 
\[
F\rightarrow E \stackrel{p}{\rightarrow} B
\]
be a fibration with the base $B$ and the total space $E$ path connected. Assume that $B$ is $m$-connected and $F$ is $n$-connected. Then there exists a partial long exact sequence 
\begin{eqnarray*}
&&0\rightarrow H^{0}(B)\rightarrow H^{0}(E)\rightarrow H^{0}(F)\rightarrow H^{1}(B)\rightarrow \cdots \\
&& \ \ \ \ \ \ \ \ \ \ \ \ \ \ \ \ \cdots\rightarrow H^{m+n}(F)\rightarrow H^{m+n+1}(B) \rightarrow H^{m+n+1}(E)\rightarrow H^{m+n+1}(F);
\end{eqnarray*}
in other word, the fibration is a cofibration up to degree $(m+n+1)$.
\end{theorem}
\begin{proof}
Let us define a homotopy commutative diagram of fibration  
\begin{gather}
\begin{aligned}
\xymatrix{
Y \ar[d]^{h}  \ar@{=}[r]  & Y \ar[d]^{\rho} \ar[r]      & \ast \ar[d]\\
F  \ar[d]^{0} \ar[r]^{i}   &  E   \ar[d]^{j}  \ar[r]^{p}    &  B  \ar@{=}[d]  \\
Z \ar[r]                        &   X   \ar[r]^{f}                   &   B,
}
\end{aligned}
\label{pfdualbmdia1}
\end{gather}
where $X$ is the homotopy cofibre of $i$, $Y$ and $Z$ is the homotopy fibre of $j$ and $f$ respectively. In order to construct the exact sequence of the theorem, we only need to estimate the connectivity of the map $f$, which is equivalent to that of the space $Z$.

We then apply Theorem \ref{BM} to the homotopy pushout and homotopy pullback diagrams
\[
\xymatrix{
F \ar[r]^{i}  \ar[d]  &  E \ar[d]    &&   Y  \pullbackcorner \ar[r]^{\rho} \ar[d] &  E \ar[d]   \\
\ast \ar[r]             &  X \pushoutcorner,              &&    \ast \ar[r]       & X ,
}
\]
to conclude that the induced map $g: F\rightarrow Y$ is $(m+n)$-connected (since $F\rightarrow \ast$ is $(n+1)$-connected and $i$ is $m$-connected). But we need to choose a nice $g$. Indeed, we may apply the functor $[F, -]$ to Diagram \ref{pfdualbmdia1} to get a commutative diagram of exact sequences of pointed sets
\begin{gather}
\begin{aligned}
\xymatrix{
[F, Y] \ar[d]^{h_\ast}  \ar@{=}[r]  & [F, Y] \ar[d]^{\rho_\ast} \ar[r]      &\ast \ar[d]\\
[F, F]  \ar[d]^{0} \ar[r]^{i_\ast}   &  [F, E]   \ar[d]^{j_\ast}  \ar[r]^{p_\ast}    &  [F, B]  \ar@{=}[d]  \\
[F, Z] \ar[r]                        &   [F, X]   \ar[r]^{f}                   &   [F, B].
}
\end{aligned}
\label{pfdualbmdia2}
\end{gather}
Then there exists a map $g: F\rightarrow Y$ such that $h\circ g= id$ and $i= i_\ast h_\ast(g)=\rho_\ast(g)=\rho\circ g$. This nice $g$ as a section of $h$ splits the long exact sequence of the homotopy groups of the fibration $h$ to direct sums
\[
\pi_i(Y)\cong \pi_i(\Omega Z)\oplus \pi_i(F).
\] 
Then $g_\ast: \pi_i(F)\rightarrow \pi_i(Y)$ is indeed an isomorphism for each $i\leq m+n$. Hence, $\Omega Z$ is $(m+n)$-connected. We should also notice that $Z$ is $0$-connected due to the commutative diagram of exact sequences
\[
\xymatrix{
\pi_1(B)   \ar[r]  \ar@{=}[d]    & \pi_0(F)  \ar[d]^{0} \ar[r]^{i_\ast}   &  \pi_0(E)=0   \ar[d]^{j_\ast}  \ar[r]^{p_\ast}    &  \pi_0(B ) \ar@{=}[d]  \\
\pi_1(B)   \ar[r]^{0}                      & \pi_0(Z) \ar[r]                        &   \pi_0(X)=0   \ar[r]^{f_\ast}                   &   \pi_0(B)=0.
}
\]
Combining the above two facts together, we see that $Z$ is $(m+n+1)$-connected, which implies that $f: X\rightarrow B$ is $(m+n+2)$-connected. Then the long exact sequence of the cohomology of the cofibration $F\stackrel{i}{\rightarrow}E \stackrel{j}{\rightarrow} X$ gives us the desired exact sequence in the theorem.
\end{proof}

\section{The Jacobi theta functions}\label{AppendixD}

A general reference for this appendix is \cite{Ch85}.

Let $$ SL_2(\mathbb{Z}):= \left\{\left.\left(\begin{array}{cc}
                                      a&b\\
                                      c&d
                                     \end{array}\right)\right|a,b,c,d\in\mathbb{Z},\ ad-bc=1
                                     \right\}
                                     $$
 as usual be the modular group. Let
$$S=\left(\begin{array}{cc}
      0&-1\\
      1&0
\end{array}\right), \ \ \  T=\left(\begin{array}{cc}
      1&1\\
      0&1
\end{array}\right)$$
be the two generators of $ SL_2(\mathbb{Z})$. Their actions on
$\mathbb{H}$ are given by
$$ S:\tau\rightarrow-\frac{1}{\tau}, \ \ \ T:\tau\rightarrow\tau+1.$$

Let
$$ \Gamma_0(2)=\left\{\left.\left(\begin{array}{cc}
a&b\\
c&d
\end{array}\right)\in SL_2(\mathbb{Z})\right|c\equiv0\ \ (\rm mod \ \ 2)\right\},$$

$$ \Gamma^0(2)=\left\{\left.\left(\begin{array}{cc}
a&b\\
c&d
\end{array}\right)\in SL_2(\mathbb{Z})\right|b\equiv0\ \ (\rm mod \ \ 2)\right\}$$

$$ \Gamma_\theta=\left\{\left.\left(\begin{array}{cc}
a&b\\
c&d
\end{array}\right)\in SL_2(\mathbb{Z})\right|\left(\begin{array}{cc}
a&b\\
c&d
\end{array}\right)\equiv\left(\begin{array}{cc}
1&0\\
0&1
\end{array}\right) \mathrm{or} \left(\begin{array}{cc}
0&1\\
1&0
\end{array}\right)\ \ (\rm mod \ \ 2)\right\}$$
be the three modular subgroups of $SL_2(\mathbb{Z})$. It is known
that the generators of $\Gamma_0(2)$ are $T,ST^2ST$, the generators
of $\Gamma^0(2)$ are $STS,T^2STS$  and the generators of
$\Gamma_\theta$ are $S$, $T^2$. (cf. \cite{Ch85}).

The four Jacobi theta-functions (c.f. \cite{Ch85}) defined by
infinite multiplications are

\be \theta(v,\tau)=2q^{1/8}\sin(\pi v)\prod_{j=1}^\infty[(1-q^j)(1-e^{2\pi \sqrt{-1}v}q^j)(1-e^{-2\pi
\sqrt{-1}v}q^j)], \ee
\be \theta_1(v,\tau)=2q^{1/8}\cos(\pi v)\prod_{j=1}^\infty[(1-q^j)(1+e^{2\pi \sqrt{-1}v}q^j)(1+e^{-2\pi
\sqrt{-1}v}q^j)], \ee
 \be \theta_2(v,\tau)=\prod_{j=1}^\infty[(1-q^j)(1-e^{2\pi \sqrt{-1}v}q^{j-1/2})(1-e^{-2\pi
\sqrt{-1}v}q^{j-1/2})], \ee
\be \theta_3(v,\tau)=\prod_{j=1}^\infty[(1-q^j)(1+e^{2\pi
\sqrt{-1}v}q^{j-1/2})(1+e^{-2\pi \sqrt{-1}v}q^{j-1/2})], \ee where
$q=e^{2\pi \sqrt{-1}\tau}, \tau\in \mathbb{H}$.

They are all holomorphic functions for $(v,\tau)\in \mathbb{C \times
H}$, where $\mathbb{C}$ is the complex plane and $\mathbb{H}$ is the
upper half plane.

Let $\theta^{'}(0,\tau)=\frac{\partial}{\partial
v}\theta(v,\tau)|_{v=0}$. The {\em Jacobi identity} \cite{Ch85},
$$\theta^{'}(0,\tau)=\pi \theta_1(0,\tau)
\theta_2(0,\tau)\theta_3(0,\tau)$$ holds.

The theta functions satisfy the the following
transformation laws (cf. \cite{Ch85}), 
\be 
\theta(v,\tau+1)=e^{\pi \sqrt{-1}\over 4}\theta(v,\tau),\ \ \
\theta\left(v,-{1}/{\tau}\right)={1\over\sqrt{-1}}\left({\tau\over
\sqrt{-1}}\right)^{1/2} e^{\pi\sqrt{-1}\tau v^2}\theta\left(\tau
v,\tau\right)\ ;\ee 
\be \theta_1(v,\tau+1)=e^{\pi \sqrt{-1}\over
4}\theta_1(v,\tau),\ \ \
\theta_1\left(v,-{1}/{\tau}\right)=\left({\tau\over
\sqrt{-1}}\right)^{1/2} e^{\pi\sqrt{-1}\tau v^2}\theta_2(\tau
v,\tau)\ ;\ee 
\be\theta_2(v,\tau+1)=\theta_3(v,\tau),\ \ \
\theta_2\left(v,-{1}/{\tau}\right)=\left({\tau\over
\sqrt{-1}}\right)^{1/2} e^{\pi\sqrt{-1}\tau v^2}\theta_1(\tau
v,\tau)\ ;\ee 
\be\theta_3(v,\tau+1)=\theta_2(v,\tau),\ \ \
\theta_3\left(v,-{1}/{\tau}\right)=\left({\tau\over
\sqrt{-1}}\right)^{1/2} e^{\pi\sqrt{-1}\tau v^2}\theta_3(\tau
v,\tau)\ .\ee

Let $\Gamma$ be a subgroup of $SL_2(\mathbb{Z}).$ A modular form over $\Gamma$ is a holomorphic function $f(\tau)$ on $\mathbb{H}\cup
\{\infty\}$ such that for any
 $$ g=\left(\begin{array}{cc}
             a&b\\
             c&d
             \end{array}\right)\in\Gamma\ ,$$
 the following property holds
 $$f(g\tau):=f\left(\frac{a\tau+b}{c\tau+d}\right)=\chi(g)(c\tau+d)^kf(\tau), $$
 where $\chi:\Gamma\rightarrow\mathbb{C}^*$ is a character of
 $\Gamma$ and $k$ is called the weight of $f$.

\end{document}